\input ./style/arxiv-vmsta.cfg
\documentclass[numbers,compress,v1.0.1]{vmsta}
\usepackage{vtexbibtags}
\usepackage{authorquery}

\usepackage[utf8]{inputenc}
\usepackage[T1]{fontenc}

\usepackage{amsmath, amsthm, amssymb}

\volume{6}
\issue{1}
\pubyear{2019}
\firstpage{109}
\lastpage{131}
\aid{VMSTA123}
\doi{10.15559/18-VMSTA123}
\articletype{research-article}

\startlocaldefs
\newcommand{\lleft}{\left}
\newcommand{\rright}{\right}

\allowdisplaybreaks

\newtheorem{thm}{Theorem}
\newtheorem{lemma}{Lemma}
\newtheorem{cor}{Corollary}

\theoremstyle{definition}
\newtheorem{remark}{Remark}

\newcommand{\RR}{\mathbb{R}}
\newcommand{\NN}{\mathbb{N}}
\newcommand{\ZZ}{\mathbb{Z}}

\newcommand{\dirac}{I}

\newcommand{\ii}{{\mathrm i}}
\newcommand{\ee}{{\mathrm e}}

\newcommand{\DD}{{\mathrm{d}}}

\newcommand{\norm}[1]{\|#1\|} 
\newcommand{\knorm}[1]{\ab{#1}_{\mathrm{K}}}
\newcommand{\localnorm}[1]{\norm{#1}_\infty} 
\newcommand{\ab}[1]{\vert#1\vert} 
\newcommand{\Ab}[1]{\Big\vert#1\Big\vert} 
\newcommand{\AB}[1]{\bigg\vert#1\bigg\vert}

\newcommand{\exponent}[1]{\exp\{#1\}}
\newcommand{\Exponent}[1]{\exp\Bigl\{#1\Bigr\}}
\newcommand{\eit}{\ee^{\ii t}}
\newcommand{\deit}{\ee^{d\ii t}}
\newcommand{\Prob}{\mathrm{P}}
\newcommand{\w}{\widehat}

\newcolumntype{d}[1]{D{.}{.}{#1}}

\endlocaldefs

\begin{document}

\begin{frontmatter}
\pretitle{Research Article}

\title{Asymptotics for the sum of three state Markov dependent random variables}

\author{\inits{G.}\fnms{Gabija}~\snm{Liaudanskait\.{e}}\thanksref{cor1}\ead[label=e2]{gabija.liaudanskaite@mif.stud.vu.lt}}
\author{\inits{V.}\fnms{Vydas}~\snm{\v{C}ekanavi\v{c}ius}\ead[label=e1]{vydas.cekanavicius@mif.vu.lt}}
\thankstext[type=corresp,id=cor1]{Corresponding author.}
\address{Faculty of Mathematics and Informatics, \institution{Vilnius University},\\ Naugardukas str. 24, LT-03225, Vilnius, \cny{Lithuania}}


\markboth{G. Liaudanskait\.{e}, V. \v{C}ekanavi\v{c}ius}{Asymptotics for the sum of three state Markov dependent random variables}

\begin{abstract}
The insurance model when the amount of claims depends on the state of the
insured person (healthy, ill, or dead) and claims are connected in a Markov
chain is investigated. The signed compound Poisson approximation is applied to
the aggregate claims distribution after $n\in\NN$ periods. The accuracy of
order $O(n^{-1})$ and $O(n^{-1/2})$ is obtained for the local and uniform
norms, respectively. In a particular case, the accuracy of estimates in total variation and
non-uniform estimates are shown to be at least
of order $O(n^{-1})$.
The characteristic function method is used. The results can be applied to
estimate the probable loss of an insurer to optimize an insurance
premium.
\end{abstract}


\received{\sday{10} \smonth{8} \syear{2018}}
\revised{\sday{11} \smonth{10} \syear{2018}}
\accepted{\sday{27} \smonth{10} \syear{2018}}
\publishedonline{\sday{19} \smonth{11} \syear{2018}}
\end{frontmatter}


\section{Introduction}

This paper is motivated by the insurance model\index{insurance models} in which the insured is
described by a random variable (rv) with three states (healthy, ill, dead),
and rvs are connected in a Markov chain.\index{Markov chain} We assume that the insurer
pays one
unit of money in the case of illness and continuously pays $d\in\NN$
units in
the case of death. We are interested in aggregate losses for the insurer
after $n\in\NN$ time periods. More precisely, let
$\xi_{0}, \xi_{1},\dots, \xi_{n},\dots$ be a non-stationary
three-state $\{a_{1},a_{2},a_{3}\}$ Markov chain.\index{Markov chain} State $a_{1}$
corresponds to being healthy, state $a_{2}$ corresponds to being ill,
and state $a_{3}$ is reached in the case of death. The insurer pays
nothing for healthy policy holders, one unit of money for the ill
individuals, and constantly pays $d$ units of money ($d\in\NN$) in the
case of death. We denote the distribution of $S_n=f(\xi_1)+\cdots+f(\xi
_n)$ $(n\in\NN)$ by $F_{n}$, that is, $\Prob(S_{n}=m)=F_{n}\{m\}$ for
$m\in\ZZ$. Here $f(a_{1}) = 0, f(a_{2})=1, f(a_{3})=d, d\in\NN$. We
will analyze a little simplified model by assuming that the probability
of a healthy person to die is equal to zero (i.e. we exclude the cases
of sudden death). Even though this assumption diminishes model's
universality, it is quite reasonable, because usually a person is ill
at least for one time period and dies only afterwards.

\def\arraystretch{1}
The matrix of transition probabilities\index{transition probabilities} $P$ is defined in the following way
\[
P= \lleft( %
\begin{array}{rrr}
1-\gamma& \gamma& 0 \\
1-\alpha-\beta& \beta& \alpha\\
0& 0& 1\\
\end{array} %
 \rright), \quad\alpha,\beta,\gamma
\in(0,1).
\]

It is assumed that at the beginning the insured person is healthy. Hence,
the initial distribution is given by
\[
\Prob(\xi_{0}=a_{1})=\pi_{1}=1, \qquad
\Prob(\xi_{0}=a_{2})=\pi_{2}=0, \qquad
\Prob(\xi_{0}=a_{3})=\pi_{3}=0.
\]
Observe, that our Markov chain\index{Markov chain} contains one absorbing state\index{absorbing state} (death).

In this paper, we consider triangular arrays\index{triangular arrays} of rvs (the scheme of series),
i.e. all transition probabilities $\alpha, \beta, \gamma$\index{transition probabilities} can depend on
$n\in\NN$. Arguably in insurance models\index{insurance models} the triangular arrays\index{triangular arrays} are more
natural than the more frequently studied less general scheme of sequences,
when it is assumed that the probability to become ill or to die does not
change as time passes.

All results are obtained under the condition
\begin{equation}\label{condition}
0<\beta\leqslant0.15, \qquad
0<\gamma\leqslant0.05, \qquad
\alpha\leqslant C_0<1, \qquad
\alpha+\beta< 1.
\end{equation}

Here $C_0\in(0,1)$ is any maximum possible value of $\alpha(n), n\in\NN$
(strictly less than~1), i.e. the maximum probability of an ill
individual to
die for all time periods $n\in\NN$.
The condition (\ref{condition}) is not
very restrictive, because $\beta\leqslant0.15$ means that the
probability to
remain ill during the next time period does not exceed 15\%, and
$\gamma\leqslant0.05$ means that the probability of a healthy person to
become ill does not exceed 5\%, that is, only chronic and epidemic illnesses
are excluded.

We denote by $C$ all positive absolute constants, and we denote by
$\theta$
any complex number satisfying $|\theta| \leqslant1$. The values of $C$ and
$\theta$ can vary from line to line or even within the same line. Sometimes,
as in (\ref{condition}), we supply constants with indices. Let $I_k$ denote
the distribution concentrated at an integer $k\in\ZZ$, and set $I=I_0$. Let
$M_{\ZZ}$ be a set of finite signed measures concentrated on $\ZZ$. The
Fourier transform and analogue of distribution function for $M\in M_{\ZZ
}$ is
denoted by $\w{M}(t)$ $(t\in\RR)$ and $M(x):=\sum_{j=-\infty}^xM\{j\}$,
respectively. Similarly, $F_n(x):=F_n\{(-\infty, x]\}$. For $y\in\RR$ and
$j\in\NN=\{1,2,3,\dots\}$, we set
\[
\binom{y}{j}:=\frac{1}{j!}y(y-1)\dots(y-j+1),\quad
\binom{y}{0}:=1.
\]
\par

If $N, M \in M_{\ZZ}$, then products and powers of $N$ and $M$ are
understood in the convolution sense, that is, for a set $A\subseteq\ZZ$,
\[
NM\{A\}=\sum_{k=-\infty}^\infty N\{A-k\}M\{k\}, \quad M^0=\dirac.
\]
The exponential of $M$ is denoted by
\[
\ee^M=\exponent{M}:=\sum_{k=0}^\infty
\frac{1}{k!}M^k.
\]
We define the local norm, the uniform (Kolmogorov) norm, and the
total-variation norm of~$M$ respectively by
\[
\localnorm{M}:=\sup_{k\in\ZZ}\ab{M\{k\}},\quad\  \knorm{M}:=\sup
_{x\in\RR}\ab{M\{(-\infty,x]\}}, \quad\  \norm{M}:=\sum
_{j=-\infty}^{\infty}\ab{M\{j\}}.
\]

In the proofs, we apply the following well-known relations:
\begin{align*}
\w{MN}(t)=\w{M}(t)\w{N}(t),&\qquad\norm{MN}\leqslant\norm{M}\norm {N},\qquad
\knorm{MN}\leqslant\norm{M}\knorm{N},
\\
\norm{MN}_{\infty}\leqslant\norm{M}\norm{N}_\infty, &\qquad\ab{\w
{M}(t)}\leqslant\norm{M}, \qquad\w{I}_a(t)=\ee^{\ii ta}, \qquad
\w{I}(t)=1.
\end{align*}
%

\section{Known results}

The compound Poisson\index{compound Poisson approximation} approximation is frequently used to approximate aggregate
losses in risk models (see, for example, \citep{DD92, Ge84, H85, Pi04, Ro07,
ZLL14}); however, in those models it is usually assumed that rvs are
independent of time period $n\in\NN$. The compound Poisson\index{compound Poisson approximation} approximation
to sums
of Markov dependent rvs was investigated in \cite{E99}. Numerous papers were
devoted to Markov Binomial distribution,\index{Markov binomial distribution} see \cite{BL06,CR09,CV10,Gan82,Aki93,XZ09,YY10}, and the references therein. It seems, however, that the case of Markov chain\index{Markov chain}
containing absorbing state\index{absorbing state} was not considered so far.
Our research is closely related to the paper \cite{SC16}, in which a
non-stationary three-state symmetric Markov chain $\xi_0, \xi_1, \dots
\xi_n,
\dots$\index{Markov chain} was investigated with the matrix of transition probabilities\index{transition probabilities}
\begin{eqnarray*}
\begin{pmatrix}
a & 1-2a & a \\
b & 1-2b & b \\
a & 1-2a & a
\end{pmatrix} %
,\quad a,b\in(0, 0.5).
\end{eqnarray*}

Let $\tilde S_n=\tilde f(\xi_1)+\cdots+\tilde f(\xi_n)$ $(n\in\NN)$,
$\tilde f(a_1)=-1$, $\tilde f(a_2)=0$, $\tilde f(a_3)=1$ and let the
initial distribution be $P(\xi_0=a_1)=\pi_{1}$, $P(\xi_0=a_2)=\pi_{2}$,
and $P(\xi_0=a_3)=\pi_{3}$. Denote the distribution of $\tilde S_n$ by
$\tilde F_n$. $\tilde G$ defines the measure with the Fourier transform:
\begin{align}\label{sliogere}
\tilde g(t)={}& \biggl(\pi_1+\frac{1-2a\cos t}{1-2a}
\pi_2+\pi_3 \biggr)\frac{1-2(a-b)}{1-2(a-b)-2a(\cos t-1)}\nonumber\\*
&\times\Exponent{\frac{2nb(1-2a)(\cos t-1)}{(1-2a+2b)(1-2a\cos t)}}.
\end{align}
As shown in \cite{SC16}, if $a,b\leqslant1/30$, then
\begin{equation}
\norm{\tilde F_n-\tilde G}\leqslant C \biggl( \min \biggl\{
\frac{1}{n}, b \biggr\}+0.2^n\ab{a-b} \biggr). \label{cit1}
\end{equation}

The main part of the approximation $\tilde G$ is a compound Poisson distribution\index{compound Poisson distribution}
with a compounding symmetrized geometric distribution. The accuracy of
approximation is at least $O(n^{-1})$. However, due to the symmetry of
distribution and possible negative values, it is difficult to find a
compatible insurance model.\index{insurance models}


\section{Measures used for approximation}

For convenience we present all Fourier transforms\index{Fourier transforms} of measures used for
construction of approximations in a separate table. Note that all measures
are denoted by the same capital letters as their Fourier transforms\index{Fourier transforms} (for
example, $\w{H}(t)$ is the Fourier transform of $H$).

The measures can be easily found from their Fourier transforms\index{Fourier transforms}
using the formula
\[
M\{k\}=\frac{1}{2\pi}\int_{-\pi}^{\pi}
\ee^{-k\ii t} \w{M}(t) \DD t \quad\text{for all }  k\in\ZZ.
\]
For example,
\[
\w{H}(t)=\frac{(1-\beta)\eit}{1-\beta\eit}.
\]

Since $\w{I}_a(t)=\ee^{\ii ta}$, for all $ k\in\ZZ$ we have
\begin{align*}
H\{k\}&=\frac{1}{2\pi}\int_{-\pi}^{\pi}
\ee^{-k\ii t} \frac{(1-\beta
)\eit}{1-\beta\eit} \DD t =\frac{1-\beta}{2\pi}\int
_{-\pi}^{\pi} \ee^{-\ii kt}\ee^{\ii t} \sum
_{j=0}^{\infty}(\beta\eit)^j \DD t
\\
&=(1-\beta) \beta^{k-1}\sum_{j=0}^{\infty}
\beta^{j-k+1}\frac{1}{2\pi
}\int_{-\pi}^{\pi}
\ee^{-kit}\ee^{(j+1)it} \DD t
\\
&=(1-\beta) \beta^{k-1}\sum_{j=0}^{\infty}
\beta^{j-k+1}I_{j+1}\{k\}
\\
&=(1-\beta)\sum_{j=0}^{\infty}
\beta^{j}I_{j+1}\{k\}.
\end{align*}

The other measures can be calculated analogously using their Fourier
transforms\index{Fourier transforms} presented \xch{in Table \ref{tab1}}{in the table below}.

\begin{table}[t!]
\def\arraystretch{3}
\caption{Fourier transforms of used measures.}\label{tab1}
\begin{tabular*}{\textwidth}{l@{\qquad\qquad\qquad }l@{}}
\hline
$\displaystyle \w{H}(t)=\frac{(1-\beta)\eit}{1-\beta\eit} $ & $\displaystyle \w{A}_1(t)=\frac{1-\beta}{1+\gamma-\beta}(\w{\Psi}(t)-1) $\\
$\displaystyle\w{\Psi}(t)=\frac{(1-\alpha-\beta)\eit}{1-\beta\eit}$ & $\displaystyle \w{A}_2(t)=-\frac{\beta(1-\beta)}{(1+\gamma-\beta)^2}(\w{H}(t)-1)(\w{\Psi}(t)-1)$\\
$\displaystyle \w{H}(t)-1=\frac{\eit-1}{1-\beta\eit}$&$\displaystyle \w{A}_3(t)=\frac{\beta^2(1-\beta)(\w{H}(t)-1)^2(\w{\Psi}(t)-1)}{(1+\gamma-\beta)^3}$\\
$\displaystyle \w{\Psi}(t)-1=\frac{(1-\alpha)\eit-1}{1-\beta\eit}$&$\displaystyle \w{A}_4(t)=-\frac{(1-\beta)^3(\w{\Psi}(t)-1)^2}{(1+\gamma-\beta)^3(1-\beta\eit)}$\\
$\displaystyle \w{U}(t)=(1-\alpha)\eit-1$ & $\displaystyle \w{A}_5(t)=\frac{3\beta(1-\beta)^3(\w{\Psi}(t)-1)^2(\w{H}(t)-1)}{(1+\gamma-\beta)^4(1-\beta\eit)}$\\
$\w{\Delta}(t)=1+\w{A}_1(t)\gamma$ &$\displaystyle \w{A}_6(t)=\frac{2(1-\beta)^5(\w{\Psi}(t)-1)^3}{(1+\gamma-\beta)^5(1-\beta\eit)^2}$\\
\multicolumn{2}{l}{$\displaystyle \w{\Delta}_1(t)=1+\w{A}_1(t)\gamma+(\w{A}_2(t)+\w{A}_4(t))\gamma^2$}\\
\multicolumn{2}{l}{$\displaystyle \w{A}(t)=1+\w{A}_1(t)\gamma+\w{A}_2(t)\gamma^2+\w{A}_3(t)\gamma^3+\w{A}_4(t)\gamma^2+\w{A}_5(t)\gamma^3+\w{A}_6(t)\gamma^3$}\\
\multicolumn{2}{l}{$\displaystyle \w{V}(t)= \frac{(\ee^{(d+1)\ii t}-1)(\beta-\gamma(1-\alpha))-(\deit-1)\w{\Delta}(t)}{(\w{A}(t)-\deit)(2\w{\Delta}(t)-1+\gamma-\beta\eit)}
+\frac{(\eit-1)[\gamma\w{\Delta}(t)-\beta+\gamma(1-\alpha)]}{(\w{A}(t)-\deit)(2\w{\Delta}(t)-1+\gamma-\beta\eit)}$}\\
\multicolumn{2}{l}{$\displaystyle \w{V}_1(t)= \frac{(\ee^{(d+1)\ii t}-1)(\beta-\gamma(1-\alpha))-(\deit-1)\w{\Delta}(t)}{(\w{\Delta}_1(t)-\deit)(2\w{\Delta}(t)-1+\gamma-\beta\eit)}
+\frac{(\eit-1)[\gamma\w{\Delta}(t)-\beta+\gamma(1-\alpha)]}{(\w{\Delta}_1(t)-\deit)(2\w{\Delta}(t)-1+\gamma-\beta\eit)}$}\\
\multicolumn{2}{l}{$\displaystyle \w{V}_2(t)= \frac{(\ee^{(d+1)\ii t}-1)(\beta-\gamma(1-\alpha))-(\deit-1)\w{\Delta}(t)}{(\w{\Delta}_1(t)-\deit)(2\w{\Delta}_1(t)-1+\gamma-\beta\eit)}
+\frac{(\eit-1)[\gamma\w{\Delta}(t)-\beta+\gamma(1-\alpha)]}{(\w{\Delta}_1(t)-\deit)(2\w{\Delta}_1(t)-1+\gamma-\beta\eit)}$}\\
\multicolumn{2}{l}{$ \w{G}(t)=\Exponent{\w{A}(t)-1-\frac{1}{2}\Big(\w{A}_1^2(t)\gamma^2+2\w{A}_1(t)(\w{A}_2(t)+\w{A}_4(t))\gamma^3\Big)+\frac{1}{3}\w{A}_1^3(t)\gamma^3}$}\\
\multicolumn{2}{l}{$\displaystyle \w{G}_1(t)=\exp\bigg\{\w{A}_1(t)\gamma+\Big(\w{A}_2(t)+\w{A}_4(t)-\frac{1}{2}\w{A}_1^2(t)\Big)\gamma^2\bigg\}$}\\
\multicolumn{2}{l}{$\displaystyle \w{E}(t)=\frac{\alpha\gamma\ee^{(n+1)d\ii t}}{(e^{(d-1)\ii t}-\beta)(\deit-(1-\gamma))-\gamma(1-\alpha-\beta)}$}\\[3pt]
\hline
\end{tabular*}
\end{table}

\section{Results}
We analyze the scheme of series, when transition probabilities\index{transition probabilities} may differ
from one time period to another time period, that is, transition probabilities\index{transition probabilities}
depend on
$n\in\NN$: $\alpha=\alpha(n), \beta=\beta(n), \gamma=\gamma(n)$.
First we formulate a general approximation result for $F_n$, where
possible smallness of $\alpha$ and $\gamma$ is taken into account.

\begin{thm}\label{T1}
Let condition (\ref{condition}) hold. Then, for all $n=1,2, \dots$,
\begin{align}
\knorm{F_n-(G^nV+E)}&\leqslant C(d+1) \biggl(
\ee^{-Cn\gamma\alpha}\sqrt {\frac{\gamma}{n}}+(\beta+4\gamma)^n
\biggr),\label{Rez1}
\\
\localnorm{F_n-(G^nV+E)}&\leqslant C(d+1)
\biggl(\frac{\ee^{-Cn\gamma
\alpha}}{n}+(\beta+4\gamma)^n \biggr)
\nonumber
.
\end{align}
\end{thm}
\begin{remark}
Observe that, since $\beta+4\gamma\leqslant0.35$, the second term in
(\ref{Rez1}) tends to zero exponentially.
\end{remark}

Unlike (\ref{sliogere}), there are two components in our approximation: the
first one contains $n$-fold convolution of a signed compound Poisson measure,\index{signed compound Poisson measure}
the second one takes into account the probability of death (the absorbing
state\index{absorbing state}). The measures of approximation are chosen in a way ensuring that the
accuracy of approximation is at least as good as in the Berry--Esseen
theorem.

\begin{cor} Let condition (\ref{condition}) hold. Then, for all
$n=1,2,\dots$,
\[
\ab{F_n-(G^nV+E)}_K\leqslant
\frac{C(d+1)}{\sqrt{n}}.
\]
\end{cor}

This accuracy is reached, when $\alpha\gamma=O(n^{-1})$. If
$\alpha,\gamma\geqslant C_1>0$, the accuracy of approximation is
exponentially sharp. That prompts a question: Is it possible to
simplify the
structure of approximation by imposing more restrictive assumptions? The
answer is positive for $\alpha$ uniformly separated from zero for all $n$.

\begin{thm}\label{T2} Let condition (\ref{condition}) hold and $\alpha
\geqslant C_2$. Then, for all $n=1,2,\dots$,
\begin{equation}
\ab{F_n-(G_1^nV_1+E
)}_K\leqslant C(d+1) \bigl(\gamma\ee^{-Cn\gamma}+(\beta +4
\gamma)^n \bigr). \label{Rez3}
\end{equation}
\end{thm}

Observe that the accuracy of approximation in (\ref{Rez3}) is at least of
order $O(n^{-1})$. This accuracy is reached if $\gamma=O(n^{-1})$.

If both probabilities are uniformly separated from zero, $F_n$ is
exponentially close to the measure $E$.

\begin{thm}\label{T3} Let condition (\ref{condition}) hold and $\alpha,
\gamma\geqslant C_2$. Then, for all $n=1,2,\dots$,
\begin{equation}
\norm{F_n-E}\leqslant C(d+1)\ee^{-Cn}.\label{Rez4}
\end{equation}
\end{thm}

Observe that, if the scheme of sequences is analyzed, all probabilities do
not depend on $n$ and hence the conditions of Theorem \ref{T3} are satisfied
as long as condition (\ref{condition}) holds. Note also that in Theorem
\ref{T3} the stronger total variation norm is used.

\begin{thm}\label{T4} Let condition (\ref{condition}) hold and $\alpha
\geqslant C_2$. Then, for all $n=1,2,\dots$,
\begin{equation}
\norm{F_n-(G_1^nV_2+E
)}\leqslant C(d+1) \bigl(\gamma\ee^{-Cn\gamma
}(1+\beta/\gamma)+n(\beta+4
\gamma)^n \bigr). \label{Rez5}
\end{equation}
\end{thm}

\begin{cor}\label{Isvada2} Let condition (\ref{condition}) hold and
$\alpha\geqslant C_2$. Then, for all $n=1,2,\dots$,
\begin{equation}
\norm{F_n-(G_1^nV_2+E
)}\leqslant\frac{C(d+1)\ee^{-Cn\gamma}}{n} \biggl(1+\frac{\beta}{\gamma}
\biggr).\label{Rez5a}
\end{equation}
\end{cor}

\begin{remark}
The local estimates in Theorem \ref{T2}, \ref{T3}, and \ref{T4} have the
same order as in (\ref{Rez3}), (\ref{Rez4}), and (\ref{Rez5}), hence we
do no
formulate them separately.
\end{remark}

In insurance models,\index{insurance models} tail probabilities are very important, see, for example
\cite{LS17,WGYCh18,YW13}. Therefore, we formulate some non-uniform
estimates for the case when $\alpha$ is uniformly separated from zero.

\begin{thm}\label{T5} Let condition (\ref{condition}) hold and $\alpha
\geqslant C_2$. Then, for any integer $k\geqslant1$ and $n\in\NN$,
\begin{eqnarray}
\ab{F_n\{k\}-(G_1^nV_2+E
)\{k\}}\!\!\!\!&\leqslant&\!\!\!\! \frac{C(d+1)\ee^{-Cn\gamma
}(\beta+\gamma)}{n(\beta+(k+1)\gamma)}.\label{Rez6}
\\
\ab{F_n(k)-(G_1^nV_2+E
) (k)}\!\!\!\!&\leqslant&\!\!\!\! \frac{Cd^2\ee^{-Cn\gamma
}}{n(1+k\gamma^2)}.\label{Rez7}
\end{eqnarray}
\end{thm}

\begin{remark}
The non-uniform estimate for distribution functions (\ref{Rez7}) is quite
inaccurate if $\gamma$ is small. On the other hand, the local non-uniform
estimate is at least of order $O(n^{-1}k^{-1})$, when $\beta$ is of the same
order as $\gamma$.
\end{remark}

When $\gamma$ is uniformly separated from zero and $\alpha$ is small,
estimate (\ref{Rez1}) could not be simplified.

\section{Auxiliary results}

We begin from the inversion inequalities.

\begin{lemma}\label{apvertimai} Let $M\in M_{\ZZ}$. Then
\begin{eqnarray}
\knorm{M}\!\!\!\!&\leqslant&\!\!\!\!\frac{1}{2\pi}\int_{-\pi}^{\pi}
\frac{\ab{\w
{M}(t)}}{\ab{\eit-1}}\DD t, \label{Tsaregradskii}
\\
\norm{M}_{\infty}\!\!\!\!&\leqslant&\!\!\!\!\frac{1}{2\pi}\int_{-\pi}^\pi
\ab{\w M(t)}\, \hbox{\rm d} t. \label{LAP}
\end{eqnarray}
If, in addition, $\sum_{k\in\ZZ}\ab{k}\ab{M\{k\}}<\infty$, then
\begin{equation}
\norm{M}\leqslant (1+b\pi )^{1/2} \biggl(\frac{1}{2\pi}\int \limits
_{-\pi}^{\pi}
\ab{\w M(t)}^2+ \frac{1}{b^2}\ab{ (\hbox{\rm
e}^{-\ii ta}\w M(t) )'}^2\hbox{\rm d}t
\biggr)^{1/2}, \label{TVAP}
\end{equation}
and, for any $a\in\RR,b>0$,
\begin{eqnarray}
\ab{k-a}\ab{M\{k\}}\!\!\!\!&\leqslant&\!\!\!\!\frac{1}{2\pi}\int_{-\pi}^{\pi}
\ab{(\w {M}(t)\ee^{-\ii ta})'}\DD t, \label{non-uniform_1}
\\
\ab{k-a}\ab{M(k)}\!\!\!\!&\leqslant&\!\!\!\!\frac{1}{2\pi}\int_{-\pi}^{\pi}
\AB{ \biggl(\frac{\w{M}(t)}{\ee^{-\ii t}-1}\ee^{-\ii ta} \biggr)'}\DD t.
\label
{non-uniform_2}
\end{eqnarray}
\end{lemma}
Observe that (\ref{Tsaregradskii}) and (\ref{non-uniform_2}) are
trivial if
integrals on the right-hand side are infinite. All inequalities are
well-known and can be found in \cite{C16} Section 6.1 and Section 6.2; see,
also \cite{Pre85} and Lemma 3.3 in \cite{SC15}.



The characteristic function method is used for the analysis of the model.
Therefore our next step is to obtain $\w F_n(t)$.

\begin{lemma}\label{lema_1}
Let condition (\ref{condition}) hold. Then the characteristic function
$\w
F_n(t)$\index{characteristic function} can be expressed in the following way:
\begin{equation}
\label{eq:a1} \w{F}_n(t)=\w{\varLambda}_1^n(t)
\w{W}_1(t) +\w{\varLambda}_2^n(t)
\w{W}_2(t) +\w{\varLambda}_3^n(t)
\w{W}_3(t).
\end{equation}

Here
\begin{eqnarray*}
\w{\varLambda}_{1,2}(t)\!\!\!\!&=&\!\!\!\!\frac{1-\gamma+\beta\eit\pm\sqrt{\w{D}(t)}}{2}, \qquad\w{
\varLambda}_{3}(t)=\deit,
\\
\w{D}(t)\!\!\!\!&=&\!\!\!\!(1-\gamma+\beta\eit)^{2} - 4\eit\bigl(\beta-\gamma(1-
\alpha)\bigr),
\\
\w{W}_{1,2}(t)\!\!\!\!&=&\!\!\!\!\frac{(\ee^{(d+1)\ii t}-1)(\beta-\gamma(1-\alpha
))-(\deit-1)\widehat{\varLambda}_{1,2}(t)}{\pm(\widehat{\varLambda
}_{1,2}(t)-\deit)\sqrt{\widehat{D}(t)}}
\\
&&\!\!\!\!{}+\frac{(\eit-1)[\gamma\widehat{\varLambda}_{1,2}(t)-\beta+\gamma(1-\alpha
)]}{\pm(\widehat{\varLambda}_{1,2}(t)-\deit)\sqrt{\widehat{D}(t)}},
\\
\w{W}_{3}(t)\!\!\!\!&=&\!\!\!\!\frac{\alpha\gamma\deit}{(e^{(d-1)\ii t}-\beta)(\deit
-(1-\gamma))-\gamma(1-\alpha-\beta)}.
\end{eqnarray*}
\end{lemma}

\begin{proof} The characteristic function $\w F_n(t)$\index{characteristic function} can be written as
follows, see \cite{SC16}:
\begin{equation}
\label{eq:w} \w{F}_n(t) = (\pi_1, \pi_2,
\pi_3) \bigl(\w{\varLambda}_1^n(t)
\vec{y}_{1}\vec {z}^T_{1}+\w{
\varLambda}_2^n(t)\vec{y}_{2}
\vec{z}^T_{2}+\w{\varLambda }_3^n(t)
\vec{y}_{3}\vec{z}^T_{3}\bigr)
(1,1,1)^T.
\end{equation}

Expression (\ref{eq:a1}) is known as Perron's formula. Similar
expression was
used for Markov binomial distribution; see, for example, \cite{CR09}.
$\w{\varLambda}_{j}(t)$ ($j=1,2,3$) are eigenvalues of the following matrix:

\def\arraystretch{1}
\[
\tilde P(t)= \lleft( %
\begin{array}{ccc}
1-\gamma& \gamma\eit& 0 \\
1-\alpha-\beta& \beta\eit& \alpha\deit\\
0& 0& \deit\\
\end{array} %
 \rright).
\]
%

We find the eigenvalues by solving the following equation:
\[
\ab{\tilde{P}(t)-\w{\varLambda}(t)\dirac}=0.
\]
It is not difficult to prove that
\begin{equation}
\label{lygtis} \w{\varLambda}_{1,2}(t)^2-\w{
\varLambda}_{1,2}(t) (1-\gamma+\beta\eit)+\eit \bigl(\beta-\gamma(1-
\alpha)\bigr)=0,
\end{equation}
and
\[
\deit- \w{\varLambda}_3(t)=0.
\]
Hence,
\begin{eqnarray*}
\w{\varLambda}_{1,2}(t)\!\!\!\!&=&\!\!\!\!\frac{1-\gamma+\beta\eit\pm
\w{D}^{1/2}(t)}{2},
\\
\w{D}(t)\!\!\!\!&=&\!\!\!\!(1-\gamma+\beta\eit)^{2} - 4\eit\bigl(\beta-\gamma(1-
\alpha)\bigr),
\\
\w{\varLambda}_{3}(t)\!\!\!\!&=&\!\!\!\!\deit.
\end{eqnarray*}
Eigenvectors $\vec{y}_j$ and $\vec{z}_j$ are obtained by solving the
following system of equations:
\begin{equation}
\label{eq:system} \lleft\{ %
\begin{array}{@{}r@{\ }c@{\ }l}
\tilde{P}(t)\vec{y}_j&=&\w{\varLambda}(t) \vec{y}_j, \\
\vec{z}_j^T\tilde{P}(t)&=&\w{\varLambda}(t) \vec{z}_j^T, \\
\vec{z}_j^T\vec{y}_j&=&1.
\end{array} %
\rright.
\end{equation}
From the first equation of system (\ref{eq:system}) we get that
$y_{j,3}=0$, hence the other two equations are equivalent because of equation
(\ref{lygtis}). Therefore,
\begin{equation}
\label{yT} \vec{y_j}^T= \biggl(y_{j,1},
\frac{1-\alpha-\beta}{\w{\varLambda
}_{j}(t)-\beta\eit}y_{j,1},0 \biggr),\quad j=1,2.
\end{equation}
Similarly, from the second equation of system  (\ref{eq:system}) we get
\begin{equation}
\label{zT} \vec{z_j}^T= \biggl(z_{j,1},
\frac{\w{\varLambda}_{j}(t)-(1-\gamma
)}{1-\alpha-\beta}z_{j,1}, \frac{\alpha\deit(\w{\varLambda}_{j}(t)
-(1-\gamma))}{(\w{\varLambda}_{j}(t)-\deit)(1-\alpha-\beta)}z_{j,1} \biggr),
\quad j=1,2.
\end{equation}
The third equation of system (\ref{eq:system}) can be written as
\begin{eqnarray}
\label{z1y1} \vec{z_j}^T\vec{y_j}\!\!\!\!&=&\!\!\!\!1,
\nonumber
\\
y_{j,1}z_{j,1}+ y_{j,2}z_{j,2}+
y_{j,3}z_{j,3}\!\!\!\!&=&\!\!\!\!1,
\nonumber
\\
y_{j,1}z_{j,1}+\frac{\w{\varLambda}_{j}(t)-(1-\gamma)}{\w{\varLambda
}_{j}(t)-\beta\eit}y_{j,1}z_{j,1}+0\!\!\!\!&=&\!\!\!\!1,
\nonumber
\\
1+\frac{\gamma\eit(1-\alpha-\beta)}{(\w{\varLambda}_{j}(t)-\beta\eit
)^2}\!\!\!\!&=&\!\!\!\!\frac{1}{y_{j,1}z_{j,1}}.
\end{eqnarray}
According to assumption, $(\pi_1,\pi_2,\pi_3)=(1,0,0)$. Substituting
(\ref{yT}), (\ref{zT}), and (\ref{z1y1}) into (\ref{eq:w}), we obtain
\begin{eqnarray*}
\w{W}_{1,2}(t)=(1,0,0)\vec{y_j}\vec{z_j}^T
\lleft( %
\begin{array}{c}
1 \\
1\\
1\\
\end{array} %
 \rright) = \frac{1+\frac{\w{\varLambda}_{j}(t)-(1-\gamma)}{1-\alpha-\beta}
(1+\frac{\alpha\deit}{\w{\varLambda}_{j}(t)-\deit} )}{1+\frac{\gamma
\eit(1-\alpha-\beta)}{(\w{\varLambda}_{j}(t)-\beta\eit)^2}},\quad
j=1,2.
\end{eqnarray*}
From equation (\ref{lygtis}) we get
\[
\frac{\w{\varLambda}_{j}(t)-(1-\gamma)}{1-\alpha-\beta}=\frac{\gamma\eit
}{\w{\varLambda}_{j}(t)-\beta\eit}.
\]
Hence,
\begin{equation}
\label{w12} \w{W}_{1,2}(t)=\frac{1+\frac{\gamma\eit}{\w{\varLambda}_{1,2}(t)-\beta\eit
} (1+\frac{\alpha\deit}{\w{\varLambda}_{1,2}(t)-\deit} )}{1+\frac
{(1-\alpha-\beta)\gamma\eit}{(\w{\varLambda}_{1,2}(t)-\beta\eit)^2}}.
\end{equation}
Applying equation (\ref{lygtis}), we prove that the numerator of
$\w{W}_{1,2}(t)$ is equal to
\begin{align}
\label{w12_nominator} &\frac{(\ee^{(d+1)\ii t}-1)(\beta-\gamma(1-\alpha))-(\deit-1)\w{\varLambda
}_{1,2}(t)}{(\w{\varLambda}_{1,2}(t)-\beta\eit)(\w{\varLambda}_{1,2}(t)-\deit
)}
\nonumber
\\
&\quad +\frac{(\eit-1)[\gamma\w{\varLambda}_{1,2}(t)-(\beta-\gamma(1-\alpha
))]}{(\w{\varLambda}_{1,2}(t)-\beta\eit)(\w{\varLambda}_{1,2}(t)-\deit)}.
\end{align}
It is easy to check that
\begin{align}
\label{Dt} (1-\gamma-\beta\eit)^2+4\gamma\eit(1-\alpha-\beta)=
\w{D}(t).
\end{align}
Similarly
\begin{align}
\label{s1} &\bigl(\widehat{\varLambda}_{1,2}(t)-
\beta\eit \bigr)^2
\nonumber
\\*
&\quad =\frac{(1-\gamma-\beta\eit)^2\pm2(1-\gamma-\beta\eit)\sqrt{\widehat
{D}(t)}+\widehat{D}(t)}{4}.
\end{align}
Using (\ref{Dt}) and (\ref{s1}), we obtain
\begin{align}
\label{s2} &\bigl(\widehat{\varLambda}_{1,2}(t)-
\beta\eit \bigr)^2+(1-\alpha-\beta)\gamma\eit
\nonumber
\\
&\quad =\frac{\sqrt{\widehat{D}(t)} (\sqrt{\widehat{D}(t)}\pm(1-\gamma
-\beta\eit) )}{2}.
\end{align}
Notice that
\[
2\bigl(\widehat{\varLambda}_{1,2}(t)-\beta\eit\bigr)=1-\gamma-\beta\eit
\pm\sqrt{\w{D}(t)}.
\]
Substituting (\ref{w12_nominator}), (\ref{s1}), and (\ref{s2}) into
(\ref{w12}), we complete the proof for $\w{\varLambda}_{1,2}$ and
$\widehat{W}_{1,2}(t)$.

Similarly, system (\ref{eq:system}) is solved with $\w{\varLambda}_{3}(t) =
\deit$. We get
\begin{gather}
\label{yT3} \vec{y_3}^T= \biggl(y_{3,1},
\frac{\deit-(1-\gamma)}{\gamma\eit
}y_{3,1}, \frac{(\deit-\beta\eit)y_{3,2}-(1-\alpha-\beta
)y_{3,1}}{\alpha\deit}y_{3,1} \biggr),\\
\label{zT3} \vec{z_3}^T=(0,0, z_{3,3}).
\end{gather}
Hence,
\begin{eqnarray}
\label{y1z3} \frac{1}{y_{3,1}z_{3,3}}=\frac{(e^{(d-1)it}-\beta)(\deit-(1-\gamma
))-\gamma(1-\alpha-\beta)}{\alpha\gamma\deit}.
\end{eqnarray}
Substituting (\ref{yT3}), (\ref{zT3}), and (\ref{y1z3}) into (\ref{eq:w}),
we get
\begin{align*}
\w{W}_{3}(t)&=(1,0,0)\vec{y_3}\vec{z_3}^T
\lleft( %
\begin{array}{c}
1 \\
1\\
1\\
\end{array} %
 \rright)=y_{3,1}z_{3,3}
\\
&=\frac{\alpha\gamma\deit}{(e^{(d-1)it}-\beta)(\deit-(1-\gamma))-\gamma
(1-\alpha-\beta)}.\qedhere
\end{align*}
\end{proof}

It is not difficult to notice that $\ab{\w{W}_3(t)}$ is equal to 1 at some
points, for example, $\w{W}_3(0)=1$, since
\begin{eqnarray*}
\w{W}_3(0)\!\!\!\!&=&\!\!\!\!\frac{\alpha\gamma}{(1-\beta)(1-(1-\gamma))-\gamma(1-\alpha
-\beta)}=\frac{\alpha\gamma}{\alpha\gamma}=1.
\end{eqnarray*}
Therefore, one cannot expect that $\w{\varLambda}_3^n(t)\w{W}_3$ be small.
Therefore we concentrate our research on possible asymptotic behavior of
other components of $\w F_n(t)$. We begin from a short expansion of
$\sqrt{\w{D}(t)}$.

Observe that $\w{D}(t)$ can be written in the following way:
\begin{eqnarray}
\label{D} \w{D}(t)=(1+\gamma-\beta\eit)^2 \biggl(1+
\frac{4\gamma((1-\alpha)\eit
-1)}{(1+\gamma-\beta\eit)^2} \biggr).
\end{eqnarray}

\begin{lemma}\label{sqrtD_short}
Let condition (\ref{condition}) hold, $\ab{ t } \leqslant\pi$. Then
\[
\sqrt{\w{D}(t)}= 1+\gamma-\beta\eit+ 5.81\theta\gamma.
\]
\end{lemma}

\begin{proof}
$\sqrt{\w{D}(t)}$ can be expanded and written as
\begin{eqnarray}
\label{sqrtD} \sqrt{\w{D}(t)} \!\!\!\!&=&\!\!\!\! (1+\gamma-\beta\eit)\sum
_{j=0}^{\infty}\binom
{1/2}{j} \biggl(\frac{4\gamma((1-\alpha)\eit-1)}{(1+\gamma-\beta\eit
)^2}
\biggr)^j
\nonumber
\\
\!\!\!\!&=&\!\!\!\!(1+\gamma-\beta\eit)+\frac{2\gamma((1-\alpha)\eit-1)}{1+\gamma-\beta
\eit}
\nonumber
\\
&&\!\!\!\!{}+\frac{16\gamma^2((1-\alpha)\eit-1)^2}{(1+\gamma-\beta\eit)^3}\sum_{j=2}^{\infty}
\binom{1/2}{j} \biggl(\frac{4\gamma((1-\alpha)\eit
-1)}{(1+\gamma-\beta\eit)^2} \biggr)^{j-2}
\nonumber
\\
\!\!\!\!&=&\!\!\!\!(1+\gamma-\beta\eit)+\frac{2\gamma((1-\alpha)\eit-1)}{1+\gamma-\beta
\eit}
\nonumber
\\
&&\!\!\!\!{}+\frac{2\theta\gamma^2\ab{(1-\alpha)\eit-1}^2}{\ab{1+\gamma-\beta\eit
}^3}\sum_{j=0}^{\infty}\AB{
\frac{4\gamma((1-\alpha)\eit-1)}{(1+\gamma
-\beta\eit)^2}}^j.
\end{eqnarray}
Observe that
\begin{gather*}
\AB{\frac{4\gamma((1-\alpha)\eit-1)}{(1+\gamma-\beta\eit)^2}} \leqslant\frac{8\cdot0.05}{(0.85+0.05)^2}\leqslant0.5,\\
\frac{\theta\gamma^2\ab{(1-\alpha)\eit-1}^2}{\ab{1+\gamma-\beta\eit
}^3}\sum_{j=0}^{\infty}\AB{
\frac{4\gamma((1-\alpha)\eit-1)}{(1+\gamma
-\beta\eit)^2}}^j\leqslant0.55\theta\gamma.
\end{gather*}
Therefore
\begin{align*}
\sqrt{\w{D}(t)} &=1+\gamma-\beta\eit+\frac{4\theta\gamma}{0.85} + 2\cdot0.55\theta
\gamma
\\
&= 1+\gamma-\beta\eit+ 5.81\theta\gamma.\qedhere
\end{align*}
\end{proof}

Next we prove that $\w\varLambda_2(t)$ is always small.
%
\begin{lemma}\label{lambda_2}
Let condition (\ref{condition}) hold, $\ab{ t } \leqslant\pi$. Then
\[
\ab{\w{\varLambda}_{2}(t)} \leqslant\beta+4\gamma.
\]
\end{lemma}
\begin{proof} From Lemma \ref{sqrtD_short} we get
\begin{align*}
\ab{ \w{\varLambda}_{2}(t) } &= \AB{\frac{1-\gamma+\beta\eit-\sqrt{\w
{D}(t)}}{2}}
\\
&= \frac{1}{2} \AB{ 1-\gamma+\beta\eit- (1+\gamma-\beta\eit+5.81\theta
\gamma)} \leqslant\beta+4\gamma.\qedhere
\end{align*}
\end{proof}

\begin{cor}\label{lambda_2_corr}
Let condition (\ref{condition}) hold, $\ab{t}\leqslant\pi$. Then
\[
\ab{\w{\varLambda}_{2}(t)} \leqslant0.35.
\]
\end{cor}


The following estimate shows that $\varLambda_1$ behaves similarly to the
compound Poisson distribution.\index{compound Poisson distribution}
%
\begin{lemma}\label{lambda_1}
Let condition (\ref{condition}) hold, $\ab{t} \leqslant\pi$. Then
\begin{eqnarray*}
\ab{\w{\varLambda}_{1}(t) }\!\!\!\!&\leqslant&\!\!\!\!1+0.4(1-\alpha)\gamma Re\bigl(
\w {H}(t)-1\bigr)-0.2\alpha\gamma
\\
\!\!\!\!&\leqslant&\!\!\!\!\exp\bigl\{0.4(1-\alpha)\gamma Re\bigl(\w{H}(t)-1\bigr)-0.2\alpha
\gamma\bigr\}.
\end{eqnarray*}
\end{lemma}
\begin{proof}

It is not difficult to check that
\begin{equation}
\label{trupmenos} \frac{1}{1+\gamma-\beta\eit}=\frac{1-\beta}{1+\gamma-\beta}\frac
{1}{1-\beta\eit}-
\frac{\beta\gamma}{1+\gamma-\beta\eit}\frac{\eit
-1}{1-\beta\eit}\frac{1}{1+\gamma-\beta}.
\end{equation}
From (\ref{sqrtD}) and (\ref{trupmenos}) it follows that
\begin{eqnarray}
\label{l1} \ab{\w{\varLambda}_{1}(t)}\!\!\!\!&=&\!\!\!\!\AB{\frac{1-\gamma+\beta\eit+\sqrt{\w
{D}(t)}}{2}}
\nonumber
\\
\!\!\!\!&\leqslant&\!\!\!\!\AB{1+\frac{\gamma(1-\beta)}{1+\gamma-\beta}\bigl(\w{\varPsi }(t)-1\bigr)}+
\frac{\beta\gamma^2}{(1+\gamma-\beta)^2}\ab{\w{\varPsi}(t)-1}\ab {\eit-1}
\nonumber
\\
&&\!\!\!\!{}+2\gamma^2\ab{\w{\varPsi}(t)-1}^2\frac{(1+\beta)^2}{(1+\gamma-\beta)^3}.
\end{eqnarray}
Notice that
\begin{gather}
\ab{\w{\varPsi}(t)}^2 =\bigl(Re\w{\varPsi}(t)\bigr)^2+
\bigl(Im\w{\varPsi}(t)\bigr)^2 \leqslant \biggl(1-
\frac{\alpha}{1-\beta} \biggr)^2\leqslant1,\nonumber\\
\label{psi_1} \ab{\w{\varPsi}(t)-1}^2 \leqslant2\bigl(1-Re\w{
\varPsi}(t)\bigr)-\frac{\alpha
}{1-\beta} \biggl(2-\frac{\alpha}{1-\beta} \biggr).
\end{gather}
For all $0\leqslant\nu\leqslant1$, we have
\begin{eqnarray}
\label{nu_psi} \ab{1+\nu(\w{\varPsi}(t)-1)} \!\!\!\!&=&\!\!\!\! \sqrt{(1-\nu)+\nu Re
\w{\varPsi}(t) + \ii \nu Im\w{\varPsi}(t)}
\nonumber
\\
\!\!\!\!&\leqslant&\!\!\!\!1+\nu(1-\nu) \bigl(Re\w{\varPsi}(t)-1\bigr).
\end{eqnarray}

Let
\[
\nu= \frac{\gamma(1-\beta)}{1+\gamma-\beta}.
\]
Substituting
(\ref{psi_1}) into (\ref{l1}) and applying inequality (\ref{nu_psi}),
we get
\begin{eqnarray*}
\ab{\w{\varLambda}_{1}(t)}\!\!\!\!&\leqslant&\!\!\!\! 1+\nu(1-\nu) \bigl(Re\w{
\varPsi}(t)-1\bigr)+\frac
{\beta\gamma^2}{(1+\gamma-\beta)^2}\ab{\w{\varPsi}(t)-1}\ab{\eit -1}
\nonumber
\\
&&\!\!\!\!{}+\frac{4\gamma^2(1+\beta)^2}{(1+\gamma-\beta)^3}\bigl(1-Re\w{\varPsi }(t)\bigr)-\frac{2\gamma^2\alpha}{1-\beta}
\frac{(1+\beta)^2}{(1+\gamma-\beta
)^3} \biggl(2-\frac{\alpha}{1-\beta} \biggr).
\end{eqnarray*}
$\ab{\w{\varPsi}(t)-1}$ can be estimated as
\[
\ab{\w{\varPsi}(t)-1} \leqslant\frac{2}{1-\beta},
\]
and $\ab{\eit-1}$ can be estimated as
\[
\ab{\eit-1}\leqslant\frac{\ab{(1-\alpha)\eit-1}}{\ab{1-\beta\eit}}\ab {1-\beta\eit}+\alpha \leqslant\ab{\w{
\varPsi}(t)-1}(1+\beta)+\alpha.
\]
Then
\begin{eqnarray*}
\label{su_psi} \ab{\w{\varLambda}_{1}(t)} \!\!\!\!&\leqslant&\!\!\!\! 1 + \bigl(Re
\w{\varPsi}(t)-1\bigr)\frac{\gamma}{1+\gamma-\beta} \biggl((1-\beta) \biggl(1-
\frac{\gamma(1-\beta)}{1+\gamma-\beta} \biggr)
\nonumber
\\
&&\!\!\!\!{}-\frac{2\gamma\beta(1+\beta)}{1+\gamma-\beta}-\frac{4\gamma(1+\beta
)^2}{(1+\gamma-\beta)^2} \biggr)
\nonumber
\\
&&\!\!\!\!{}+\frac{2\alpha\gamma^2}{(1-\beta)(1+\gamma-\beta)} \biggl(\frac{\beta
}{1+\gamma-\beta}-\frac{(1+\beta)^2}{(1+\gamma-\beta)^2} \biggl(2-
\frac
{\alpha}{1-\beta} \biggr) \biggr).
\end{eqnarray*}

Notice that
\[
Re\w{\varPsi}(t)-1=(1-\alpha)Re\bigl(\w{H}(t)-1\bigr)-\frac{\alpha-\alpha\beta\cos
(t)}{\ab{1-\beta\eit}^2}.
\]

Finally,
\begin{align}
\ab{\w{\varLambda}_{1}(t)}&\leqslant 1+Re\bigl(\w{H}(t)-1\bigr)
\frac{(1-\alpha
)\gamma}{1+\gamma-\beta} \biggl((1-\beta) \biggl(1-\frac{\gamma(1-\beta
)}{1+\gamma-\beta} \biggr)
\nonumber
\\
&\quad {}-\frac{2\gamma\beta(1+\beta)}{1+\gamma-\beta}-\frac{4\gamma(1+\beta
)^2}{(1+\gamma-\beta)^2} \biggr)
\nonumber
\\
&\quad {}-\frac{\alpha\gamma}{1+\gamma-\beta} \biggl[\frac{1-\beta\cos(t)}{\ab
{1-\beta\eit}^2} \biggl((1-\beta) \biggl(1-
\frac{\gamma(1-\beta)}{1+\gamma
-\beta} \biggr)
\nonumber
\\
&\quad {}-\frac{2\gamma\beta(1+\beta)}{1+\gamma-\beta}-\frac{4\gamma(1+\beta
)^2}{(1+\gamma-\beta)^2} \biggr)
\nonumber
\\
&\quad {}-\frac{2\gamma}{1-\beta} \biggl(\frac{\beta}{1+\gamma-\beta}-\frac
{(1+\beta)^2}{(1+\gamma-\beta)^2} \biggl(2-
\frac{\alpha}{1-\beta} \biggr) \biggr) \biggr]
\nonumber
\\
&\leqslant1+0.4(1-\alpha)\gamma Re\bigl(\w{H}(t)-1\bigr)-0.2\alpha\gamma
\nonumber
\\
&\leqslant\exp\bigl\{0.4(1-\alpha)\gamma Re\bigl(\w{H}(t)-1\bigr)-0.2\alpha
\gamma\bigr\}.
\end{align}
\end{proof}

%
\begin{cor}\label{lambda_1_corr}
Let condition (\ref{condition}) hold, $\ab{ t } \leqslant\pi$. Then
\[
\ab{\w{\varLambda}_1(t)}\leqslant1+C\gamma\bigl(Re\w{H}(t)-1-
\alpha\bigr) \leqslant\exp\bigl\{C\gamma\bigl(Re\w{H}(t)-1-\alpha\bigr)\bigr\}.
\]
\end{cor}

%
Next we demonstrate that $\ab{\w{W}_2(t)}$ is always small.
%
\begin{lemma}\label{w2}
Let condition (\ref{condition}) hold, $\ab{ t } \leqslant\pi$. Then
\[
\ab{\w{W}_2(t)}\leqslant2(d+1)\ab{\eit-1}.
\]
\end{lemma}
\begin{proof}
From Lemma \ref{sqrtD_short} we have
\begin{equation}
\label{sqrtD2} \Ab{\sqrt{\w{D}(t)}}\geqslant1+\gamma-\beta-5.81\gamma\geqslant
1-4.81\cdot0.05-0.15\geqslant0.6.
\end{equation}
By applying Corollary \ref{lambda_2_corr}, we get
\begin{equation}
\label{l_ekit} \ab{\w{\varLambda}_2(t)-\deit}\geqslant1-\ab{\w{
\varLambda}_2(t)}\geqslant 1-0.35=0.65.
\end{equation}
Hence,
\begin{eqnarray}
\label{abw2_1} \Ab{\w{W}_2(t)}\!\!\!\!&\leqslant&\!\!\!\!\frac{(d+1)\ab{\eit-1} (2\ab{\beta-\gamma
(1-\alpha)}+(1+\gamma)\ab{\w{\varLambda}_2(t)} )}{0.65\cdot
0.6}
\nonumber
\\
\!\!\!\!&\leqslant&\!\!\!\!\frac{(d+1)\ab{\eit-1} (2\max\{\beta, \gamma(1-\alpha)\}
+(1+\gamma)\cdot0.35 )}{0.39}
\nonumber
\\
\!\!\!\!&\leqslant&\!\!\!\!2(d+1)\ab{\eit-1}.
\end{eqnarray}
\end{proof}

To approximate $\ab{\w{W}_1(t)}$, we need a longer expansion for
$\sqrt{\w{D}(t)}$.


\begin{lemma}\label{sqrtD_long}
Let condition (\ref{condition}) hold, $\ab{t}\leqslant\pi$. Then
\[
\sqrt{\w{D}(t)}=2\w{A}(t)-1+\gamma-\beta\eit+C\theta\gamma^4\bigl(
\bigl(1-Re\w {H}(t)\bigr)^2+\alpha^4\bigr).
\]
If also $\alpha\geqslant C_2$, then
\[
\Bigl(\sqrt{\w{D}(t)}\Bigr)'=\bigl(2\w{\varDelta}_1(t)-1+
\gamma-\beta\eit\bigr)'+C\theta\gamma^3.
\]
\end{lemma}

\begin{proof}
The expansion of $\w{D}(t)$ follows from equations (\ref{D}) and (\ref{trupmenos}).
The second equation of this lemma is proved similarly.
\end{proof}

\begin{cor}\label{sqrtD_long_corr}
Let condition (\ref{condition}) hold, $\ab{t}\leqslant\pi$. Then
\[
\w{\varLambda}_1(t)=\w{A}(t)+C\theta\gamma^4\bigl(
\bigl(1-Re\w{H}(t)\bigr)^2+\alpha^4\bigr).
\]
\end{cor}

\begin{cor}\label{sqrtD_long_corr_2}
Let condition (\ref{condition}) hold, $\alpha\geqslant C_2$,
$\ab{t}\leqslant\pi$. Then
\[
\w{\varLambda}_1(t)=1+\w{A}_1(t)\gamma+ \bigl(
\w{A}_2(t)+\w{A}_4(t)\bigr)\gamma ^2+C\theta
\gamma^3.
\]
\end{cor}

The following three lemmas are needed for the approximation of ${W}_1$.


\begin{lemma}\label{At}
Let condition (\ref{condition}) hold, $\ab{t}\leqslant\pi$. Then
\[
\ab{\w{A}(t)}\leqslant1+C\gamma\bigl(Re\w{H}(t)-1-\alpha\bigr).
\]
If also $\alpha\geqslant C_2$, then there exists C such that
\[
\ab{\w{\varDelta}_1(t)}\leqslant1-C\gamma.
\]
\end{lemma}

\begin{proof}
The proof is very similar to the proof of Lemma \ref{lambda_1} and,
therefore, is omitted.
\end{proof}

%
\begin{lemma}\label{w1}
Let condition (\ref{condition}) hold, $\ab{ t } \leqslant\pi$. Then
\[
\ab{\w{W}_1(t)-\w{V}(t)}\leqslant C(d+1)\gamma\ab{\eit-1}.
\]
\end{lemma}
\begin{proof}
From Corollary \ref{lambda_1_corr} and Lemma \ref{At} it follows that
\begin{eqnarray}
\ab{\w{\varLambda}_1(t)-\deit}\!\!\!\!&\geqslant&\!\!\!\! C\gamma\bigl(1-Re
\w{H}(t)+\alpha \bigr),\label{l-deit}
\\
\ab{\w{A}(t)-\deit}\!\!\!\!&\geqslant&\!\!\!\! C\gamma\bigl(1-Re\w{H}(t)+\alpha
\bigr).\label{A-deit}
\end{eqnarray}

Applying (\ref{sqrtD2}), (\ref{l-deit}), (\ref{A-deit}), Lemma
\ref{sqrtD_long} and Corollary \ref{sqrtD_long_corr}, the result follows.
\end{proof}

%
\begin{lemma}\label{w1_a}
Let condition (\ref{condition}) hold, $\alpha\geqslant C_2$, $\ab{ t }
\leqslant\pi$. Then
\[
\ab{\w{W}_1(t)-\w{V}_1(t)}\leqslant C(d+1)\gamma\ab{
\eit-1}.
\]
\end{lemma}
\begin{proof}
Since $\alpha\geqslant C_2$,
\[
\ab{\w{\varLambda}_1(t)-\deit}\geqslant C\gamma\bigl(1-Re
\w{H}(t)+\alpha \bigr)\geqslant C\gamma(0+C_2)\geqslant C\gamma.
\]
From Corollary \ref{sqrtD_long_corr_2} it follows that
\begin{align*}
\ab{\w{\varLambda}_1(t)-\w{\varDelta}_1} =C
\gamma^3.
\end{align*}
Also, from Lemma \ref{At} it follows that
\begin{equation}
\label{V1_vard} \ab{\w{\varDelta}_1-\deit} \geqslant1 - (1-C\gamma) = C
\gamma.
\end{equation}
Hence, it is easy to check that the inequality of the lemma is correct.
\end{proof}

%
\begin{lemma}\label{w1-v1}
Let condition (\ref{condition}) hold, $\ab{ t } \leqslant\pi$. Then
\begin{eqnarray}
\int_{-\pi}^{\pi}\frac{\ab{\w{\varLambda}_1(t)}^n\ab{\w{W}_1(t)-\w
{V}(t)}}{\ab{\eit-1}} \DD t\!\!\!\!&\leqslant&\!\!\!\!  C(d+1)\sqrt{\frac{\gamma}{n}}\ee^{-Cn\gamma\alpha},
\\
\int_{-\pi}^{\pi}\ab{\w{\varLambda}_1(t)}^n
\ab{\w{W}_1(t)-\w{V}(t)} \DD t\!\!\!\!&\leqslant&\!\!\!\! C(d+1)
\frac{\ee^{-Cn\gamma\alpha}}{n}.
\end{eqnarray}
\end{lemma}
\begin{proof}

It is obvious that
\begin{eqnarray}
\label{ReH-1} Re\w{H}(t)-1=\frac{(1+\beta)(\cos(t)-1)}{\ab{1-\beta\eit}^2}=-2C\sin^2(t/2).
\end{eqnarray}

We will use the following simple inequality
%
\begin{equation}
\label{exp_sin} \int_{-\pi}^{\pi}\ab{
\sin(t/2)}^k\exp\bigl\{-2\lambda\sin^2(t/2)\bigr\}\DD t
\leqslant C(k)\lambda^{-(k+1)/2}.\vadjust{\goodbreak}
\end{equation}

By applying Lemma \ref{lambda_1}, Lemma \ref{w1}, (\ref{ReH-1}), and
(\ref{exp_sin}), we get
\begin{align*}
&\int_{-\pi}^{\pi}
\frac{\ab{\w{\varLambda
}_1(t)}^n\ab{\w{W}_1(t)-\w{V}(t)}}{\ab{\eit-1}} \DD t
\\
&\quad \leqslant\int_{-\pi}^{\pi}C(d+1)\gamma\exp\bigl\{n
\bigl(0.4(1-\alpha)\gamma\bigl(Re \w{H}(t)-1\bigr)-0.2\gamma\alpha\bigr)\bigr\}
\DD t
\\
&\quad \leqslant\int_{-\pi}^{\pi}C(d+1)\gamma\exp\bigl\{C n
\gamma\bigl(Re\w{H}(t)-1\bigr)\bigr\} \ee^{-Cn\gamma\alpha}\DD t
\\
&\quad \leqslant C(d+1)\sqrt{\frac{\gamma}{n}}\ee^{-Cn\gamma\alpha}.
\end{align*}

The second inequality of the lemma is proved similarly.
\end{proof}

%
\begin{lemma}\label{w1-v1_a}
Let condition (\ref{condition}) hold, $\alpha\geqslant C_2$, $\ab{ t }
\leqslant\pi$. Then
\[
\int_{-\pi}^{\pi}\frac{\ab{\w{\varLambda}_1(t)}^n\ab{\w{W}_1(t)-\w
{V}_1(t)}}{\ab{\eit-1}} \DD t\leqslant
C(d+1)\gamma\ee^{-Cn\gamma}.
\]
\end{lemma}

\begin{proof}
From Lemma \ref{lambda_1} and Lemma \ref{w1_a} it follows that
\begin{align*}
\int_{-\pi}^{\pi}\frac{\ab{\w{\varLambda}_1(t)}^n\ab{\w{W}_1(t)-\w
{V}_1(t)}}{\ab{\eit-1}} \DD t &
\leqslant\int_{-\pi}^{\pi}C(d+1)\gamma\exp
\{-0.2C_2\gamma n\}\DD t
\\
&\leqslant C(d+1)\gamma\ee^{-Cn\gamma}.\qedhere
\end{align*}
\end{proof}

%
\begin{lemma}\label{G}
Let condition (\ref{condition}) hold, $\ab{ t } \leqslant\pi$. Then
\begin{eqnarray}
\int_{-\pi}^{\pi}\frac{\ab{\w{V}(t)}\ab{\w{\varLambda}_1^n(t)-\w
{G}^n(t)}}{\ab{\eit-1}} \DD t\!\!\!\!&\leqslant&\!\!\!\!
C(d+1)\gamma\sqrt{\frac{\gamma}{n}}\ee^{-Cn\gamma\alpha},
\\
\int_{-\pi}^{\pi}\ab{\w{V}(t)}\ab{\w{
\varLambda}_1^n(t)-\w{G}^n(t)} \DD t
\!\!\!\!&\leqslant&\!\!\!\! C(d+1)\frac{\gamma}{n}\ee^{-Cn\gamma\alpha}.
\end{eqnarray}
\end{lemma}
\begin{proof}
Notice that
\begin{eqnarray}
\label{V} \ab{\w{V}(t)}\!\!\!\!&\leqslant&\!\!\!\!\frac{C(d+1)\ab{\eit-1}}{\gamma(1-Re\w
{H}(t)+\alpha)},
\\
\ab{\w{\varLambda}_1^n(t)-\w{G}^n(t)}
\!\!\!\!&\leqslant&\!\!\!\!\ab{\w{\varLambda}_1(t)-\w {G}(t)}\cdot n \cdot\max\bigl\{
\ab{\w{\varLambda}_1(t)}^{n-1},\ab{\w {G}(t)}^{n-1}
\bigr\}.
\nonumber
\end{eqnarray}

From Corollary \ref{lambda_1_corr} we have
$\ab{\w{\varLambda}_1}\leqslant\exp\{C\gamma(Re\w{H}(t)-1-\alpha)\}$. Taking
into account that $\ab{\ee^{a+b\ii}}=\ee^a$, $\ab{\w{G}(t)}$ can be estimated
as
\[
\ab{\w{G}(t)}\leqslant\exp\bigl\{C\gamma\bigl(Re\w{H}(t)-1-\alpha\bigr)\bigr
\}.
\]

Using Corollary \ref{sqrtD_long_corr}, we have that
\begin{eqnarray}
\label{l-g} \ab{\w{\varLambda}_1(t)-\w{G}(t)}\!\!\!\!&=&\!\!\!\!\ab{\exp\{
\ln\w{\varLambda}_1(t)\}-\exp\{ \ln\w{G}(t)\}}
\nonumber
\\
\!\!\!\!&\leqslant&\!\!\!\! C\ab{\ln\w{\varLambda}_1(t)-\ln\w{G}(t)}
\nonumber
\\
\!\!\!\!&=&\!\!\!\!C\AB{\bigl(\w{\varLambda}_1(t)-1\bigr)-\frac{(\w{\varLambda}_1(t)-1)^2}{2}
\nonumber
\\
&&\!\!\!\!{}+\frac{(\w{\varLambda}_1(t)-1)^3}{3}+\frac{C\theta\ab{\w{\varLambda
}_1(t)-1}^4}{4}-\ln\w{G}(t)}
\nonumber
\\
\!\!\!\!&=&\!\!\!\!C\AB{\bigl(\w{A}(t)-1\bigr)-\frac{1}{2} \bigl(\w{A}_1^2(t)
\gamma^2+2\w{A}_1(t) \bigl(\w {A}_2(t)+
\w{A}_4(t)\bigr)\gamma^3 \bigr)
\nonumber
\\
&&\!\!\!\!{}+\frac{1}{3}\w{A}_1^3(t)\gamma^3+C
\theta\gamma^4\bigl(\bigl(1-Re\w {H}(t)\bigr)^2+
\alpha^4\bigr)-\ln{\w{G}(t)}}
\nonumber
\\
\!\!\!\!&\leqslant&\!\!\!\!C\gamma^4\bigl(\bigl(1-Re\w{H}(t)\bigr)^2+
\alpha^4\bigr).
\end{eqnarray}

By applying (\ref{V}), (\ref{l-g}), and the inequality $x\ee
^{-x}\leqslant1$,
for all $x>0$, we can estimate the following integral:
\begin{align}
&\int_{-\pi}^{\pi}
\frac{\ab{\w{V}(t)}\ab{\w
{\varLambda}_1^n(t)-\w{G}^n(t)}}{\ab{\eit-1}} \DD t
\nonumber
\\
&\quad \leqslant C(d+1)\int_{-\pi}^{\pi}n\exp\bigl\{nC\gamma
\bigl(Re\w{H}(t)-1-\alpha \bigr)\bigr\}\gamma^3\bigl(\bigl(1-Re
\w{H}(t)\bigr)+1\bigr)\DD t
\nonumber
\\
&\quad \leqslant C(d+1)\int_{-\pi}^{\pi}n
\gamma^3\frac{\exp\{n\cdot0.5C\gamma
(Re\w{H}(t)-1)\}}{n\cdot0.5C\gamma(-Re\w{H}(t)+1)}\ee^{-Cn\gamma
\alpha}\bigl(2-Re\w{H}(t)
\bigr)\DD t
\nonumber
\\
&\quad \leqslant C(d+1)\int_{-\pi}^{\pi}\gamma^2
\exp\bigl\{-2Cn\gamma\sin^2(t/2)\bigr\} \ee^{-Cn\gamma\alpha}\DD t
\nonumber
\\
&\quad \leqslant C(d+1)\gamma\sqrt{\frac{\gamma}{n}}\ee^{-Cn\gamma\alpha}.
\end{align}

The second inequality of this lemma is proved similarly.
\end{proof}

%
\begin{lemma}\label{G_a}
Let condition (\ref{condition}) hold, $\alpha\geqslant C_2$, $\ab{ t }
\leqslant\pi$. Then
\[
\int_{-\pi}^{\pi}\frac{\ab{\w{V}_1(t)}\ab{\w{\varLambda}_1^n(t)-\w
{G}_1^n(t)}}{\ab{\eit-1}} \DD t\leqslant
C(d+1)\gamma\ee^{-Cn\gamma}.
\]
\end{lemma}
\begin{proof}
Since $\alpha\geqslant C_2$,
\begin{equation}
\label{V1} \ab{\w{V}_1(t)}\leqslant\frac{C(d+1)\ab{\eit-1}}{\gamma},
\end{equation}
and
\begin{equation}
\ab{\w{\varLambda}_1^n(t)-\w{G}_1^n(t)}
\leqslant\ab{\w{\varLambda}_1(t)-\w {G}_1(t)}\cdot n
\cdot\exp\bigl\{-C\gamma(n-1)\bigr\}.
\end{equation}
$\ab{\w{\varLambda}_1(t)-\w{G}_1(t)}$ is estimated by applying Corollary
\ref{sqrtD_long_corr_2}:
\begin{eqnarray}
\label{l-g_a} \ab{\w{\varLambda}_1(t)-\w{G}_1(t)}
\!\!\!\!&\leqslant&\!\!\!\! C\ab{\ln\w{\varLambda}_1(t)-\ln \w{G}_1(t)}
\nonumber
\\
\!\!\!\!&=&\!\!\!\!C\AB{\bigl(\w{\varLambda}_1(t)-1\bigr)-\frac{(\w{\varLambda}_1(t)-1)^2}{2}+
\frac
{C\theta\ab{\w{\varLambda}_1(t)-1}^3}{3}-\ln\w{G}_1(t)}
\nonumber
\\
\!\!\!\!&=&\!\!\!\!C\AB{\w{A}_1(t)\gamma+\bigl(\w{A}_2(t)+
\w{A}_4(t)\bigr)\gamma^2-\frac{1}{2}\w
{A}_1^2(t)\gamma^2
\nonumber
\\*
&&\!\!\!\!{}+C\theta\gamma^3-\ln{\w{G}_1(t)}}
\nonumber
\\*
\!\!\!\!&\leqslant&\!\!\!\!C\gamma^3.
\end{eqnarray}

By applying (\ref{V1}), (\ref{l-g_a}), and the inequality
$x\ee^{-x}\leqslant1$, for all $x>0$, we can estimate the following integral:
\begin{align*}
\int_{-\pi}^{\pi}\frac{\ab{\w{V}_1(t)}\ab{\w{\varLambda}_1^n(t)-\w
{G}_1^n(t)}}{\ab{\eit-1}} \DD t &\leqslant
C(d+1)\int_{-\pi}^{\pi}n\gamma^2\exp\{-nC
\gamma\}\DD t
\\
&\leqslant C(d+1)\int_{-\pi}^{\pi}n
\gamma^2\frac{\exp\{-n0.5C\gamma\}
}{n 0.5C\gamma}\DD t
\\
&\leqslant C(d+1)\gamma\ee^{-Cn\gamma}.\qedhere
\end{align*}
\end{proof}

%
\begin{lemma}\label{nelygybes}
Let condition (\ref{condition}) hold, $\alpha\geqslant C_2$,
$\ab{t}\leqslant\pi$. Then
\begin{align*}
\ab{\w W_1(t)}&\leqslant\frac{C(d+1)}{\gamma},\qquad&\ab{\w
W_1'(t)} &\leqslant\frac{C(d+1)(1+\beta/\gamma)}{\gamma},
\\
\ab{\w W_2(t)}&\leqslant C(d+1),\qquad&\ab{\w W_2'(t)}&
\leqslant C(d+1),
\\
\ab{\w V_2(t)}&\leqslant\frac{C(d+1)}{\gamma},\qquad&\ab{\w
V_2'(t)}&\leqslant\frac{C(d+1)(1+\beta/\gamma)}{\gamma},
\\
\ab{\w W_1(t)-\w V_2(t)}&\leqslant C(d+1)\gamma,\quad&
\ab{\w W_1'(t)-\w V_2'(t)}&
\leqslant C(d+1)\gamma(1+\beta/\gamma),
\\
\ab{\w\varLambda_1(t)}&\leqslant\ee^{-C\gamma},\quad&\ab{\w
G_1(t)}&\leqslant\ee^{-C\gamma},
\\
\ab{\w\varLambda_1'(t)}&\leqslant C\gamma,\quad&\ab{
\w G_1'(t)}&\leqslant C\gamma,
\\
\ab{\w\varLambda_2(t)}&\leqslant\beta+4\gamma,\quad&\ab{\w\varLambda
_2'(t)}&\leqslant C(\beta+4\gamma),
\\
\ab{\w\varLambda_1(t)-\w G_1(t)}&\leqslant C
\gamma^3,\quad&\ab{(\w\varLambda _1^n(t)-
\w G_1^n(t))'}&\leqslant C
\gamma^2\ee^{-Cn\gamma},
\\
\frac{\ab{1-\ee^{d\ii t}}}{\ab{\w\varLambda_1(t)-\ee^{d\ii t}}}&\leqslant C, \quad&\frac{\ab{1-\ee^{d\ii t}}}{\ab{\w{\varDelta}_1(t)-\ee^{d\ii
t}}}&\leqslant C.
\end{align*}
\end{lemma}
\begin{proof}
All inequalities are based on the previously obtained estimates of
$\ab{\w{\varLambda}_1(t)}$, $\ab{\w{\varLambda}_2(t)}$, $\ab{\w W_2(t)}$, $\ab
{\w
G_1(t)}$, and the expansion of $\sqrt{\w{D}(t)}$. The inequalities containing
$\w{V}_2(t)$ are proved similarly to those of $\w{V}_1(t)$ (see Lemma
\ref{w1_a}).
\end{proof}

\section{Proofs}


\begin{proof}[Proof of Theorem \ref{T1}]
Applying inversion formula\index{inversion formula} (\ref{Tsaregradskii}), Lemma \ref{w1-v1}, and
Lemma \ref{G} we prove
\begin{align*}
\hspace{2em}&\hspace{-2em}\knorm{F_n-(G^nV+E)}
\\
&\leqslant\frac{1}{2\pi}\int_{-\pi}^{\pi}
\frac{\ab{\w{F}_n(t)-\w
{G}^n(t)\w{V}(t)-\w{E}(t)}}{\ab{\eit-1}}\DD t
\\
&\leqslant\frac{1}{2\pi}\int_{-\pi}^{\pi}
\frac{\ab{\w{\varLambda
}_1^n(t)}\ab{\w{W}_1(t)-\w{V}(t)}}{\ab{\eit-1}}\DD t+\frac{1}{2\pi}\int_{-\pi}^{\pi}
\frac{\ab{\w{V}(t)}\ab{\w{\varLambda}_1^n(t)-\w{G}^n(t)}}{\ab
{\eit-1}}\DD t
\\
&\quad+\frac{1}{2\pi}\int_{-\pi}^{\pi}
\frac{\ab{\w{\varLambda}_2^n(t)\w
{W}_2(t)}}{\ab{\eit-1}}\DD t
\\
&\leqslant C(d+1)\sqrt{\frac{\gamma}{n}}\ee^{-Cn\gamma\alpha
}+C(d+1) (\beta+4
\gamma)^n.
\end{align*}

The local estimate is obtained analogously by applying inversion formula\index{inversion formula}
(\ref{LAP}).
\end{proof}


\begin{proof}[Proof of Theorem \ref{T2}]
The proof is similar to the proof of Theorem \ref{T1}.
Lemma \ref{w1-v1_a} and Lemma \ref{G_a} are applied instead of Lemma
\ref{w1-v1} and Lemma \ref{G}, since $\alpha\geqslant C_2$.
\end{proof}


\begin{proof}[\bf{Proof of Theorem \ref{T3}}]
Taking into account Corollary \ref{lambda_2_corr} and Lemma \ref{nelygybes},
we get
\begin{eqnarray*}
\ab{\w{\varLambda}_{1,2}^n \w{W}_{1,2}}
\!\!\!\!&\leqslant&\!\!\!\!C(d+1)\ee^{-Cn},
\\
\ab{(\w{\varLambda}_{1,2}^n \w{W}_{1,2}
)'}\!\!\!\!&\leqslant&\!\!\!\!\ab{(\w{\varLambda }_{1,2}^n
)'}\ab{\w{W}_{1,2}}+\ab{\w{\varLambda}_{1,2}^n}
\ab{\w{W}_{1,2}'}
\\
\!\!\!\!&\leqslant&\!\!\!\! nC(d+1)\ee^{-C(n-1)}+C(d+1)\ee^{-Cn}
\\
\!\!\!\!&\leqslant&\!\!\!\! C(d+1)n\ee^{-Cn}.
\end{eqnarray*}

From inversion formula\index{inversion formula} (\ref{TVAP}) applied with $a=0$ and $b=1$ we get
\begin{align*}
\norm{F_n-E}&= \norm{\varLambda_1^n
W_1 + \varLambda_2^n W_2}
\leqslant\norm {\varLambda_1^n W_1} + \norm{
\varLambda_2^n W_2}
\\
&\leqslant (1+\pi)^{1/2} \biggl(\frac{1}{2\pi}\int\limits
_{-\pi}^{\pi
}
\ab{\w{\varLambda}_1^n \w{W}_1}^2
+\ab{(\w{\varLambda}_1^n \w{W}_1
)'}^2 \hbox{\rm d}t \biggr)^{1/2}
\\
&\quad {}+ (1+\pi)^{1/2} \biggl(\frac{1}{2\pi}\int\limits
_{-\pi}^{\pi}
\ab{\w {\varLambda}_2^n \w{W}_2}^2+
\ab{(\w{\varLambda}_2^n \w{W}_2
)'}^2\hbox{\rm d}t \biggr)^{1/2}
\\
&\leqslant C(d+1)\ee^{-Cn}.\qedhere
\end{align*}
\end{proof}


\begin{proof}[Proof of Theorem \ref{T4}]
\begin{eqnarray*}
\norm{F_n-(G_1^nV_2+E
)}\!\!\!\!&\leqslant&\!\!\!\!\norm{(\varLambda_1^n-G_1^n
)W_1}+\norm {G_1^n(W_1-V_2)}+
\norm{\varLambda_2^nW_2}.
\end{eqnarray*}

From Lemma \ref{nelygybes}, we get
\begingroup
\allowdisplaybreaks
\begin{align*}
\ab{\w{\varLambda}_2^n(t)\w{W}_2(t)}&
\leqslant C(d+1) (\beta+4\gamma)^n,
\\
\ab{(\w{\varLambda}_2^n(t)\w{W}_2(t)
)'}&\leqslant\ab{(\w{\varLambda }_2^n(t)
)'\w{W}_2(t)}+\ab{\w{\varLambda}_2^n(t)
\w{W}_2'(t)}
\\
&\leqslant C(d+1)n(\beta+4\gamma)^n+C(d+1) (\beta+4
\gamma)^n
\\
&\leqslant C(d+1)n(\beta+4\gamma)^n,
\\
\ab{\w{G}_1^n(t) (\w{W}_1(t)-
\w{V}_2(t))}&\leqslant C(d+1)\gamma\ee ^{-Cn\gamma},
\\
\ab{(\w{G}_1^n(t) (\w{W}_1(t)-
\w{V}_2(t)))'}&\leqslant\ab{(\w
{G}_1^n(t))'(
\w{W}_1(t)-\w{V}_2(t))}+\ab{\w{G}_1^n(t)
(\w{W}_1(t)-\w {V}_2(t))'}
\\*
&\leqslant C(d+1)n\gamma^2\ee^{-C(n-1)\gamma}+C(d+1)\gamma
\ee^{-Cn\gamma
}(1+\beta/\gamma)
\\
&\leqslant C(d+1)\gamma\ee^{-Cn\gamma}(1+\beta/\gamma),
\\
\ab{(\w{\varLambda}_1^n(t)-\w{G}_1^n(t)
)\w{W}_1(t)}&\leqslant n\ab{\w {\varLambda}_1(t)-
\w{G}_1(t)}\ee^{-C(n-1)\gamma}\frac{C(d+1)}{\gamma}
\\
&\leqslant C(d+1)\gamma\ee^{-Cn\gamma},
\\
\ab{((\w{\varLambda}_1^n(t)-
\w{G}_1^n(t))\w{W}_1(t)
)'}&\leqslant\ab{(\w {\varLambda}_1^n(t)-
\w{G}_1^n(t))'\w{W}_1(t)}+
\ab{(\w{\varLambda}_1^n(t)-\w {G}_1^n(t)
)\w{W}_1'(t)}\\
&\leqslant C(d+1)\gamma\ee^{-Cn\gamma}(1+\beta/\gamma).
\end{align*}
\endgroup

By applying inversion formula\index{inversion formula} (\ref{TVAP}) with $a=0$ and $b=1$, we prove
\begin{align*}
\norm{F_n-(G_1^nV_2+E
)}&\leqslant C(d+1) \bigl(\gamma\ee^{-Cn\gamma}(1+\beta /\gamma)+n(\beta+4
\gamma)^n\bigr).\qedhere
\end{align*}
\end{proof}


\begin{proof}[Proof of Theorem \ref{T5}]

We use the inequalities obtained in the proof of Theorem \ref{T4} and
inversion formula\index{inversion formula} (\ref{non-uniform_1}) with $a=0$. We have
\begin{align*}
\hspace{2em}&\hspace{-2em} k\ab{F_n-(G_1^nV_2+E
)\{k\}}
\\
&\leqslant\frac{1}{2\pi}\int_{-\pi}^{\pi}\ab{
(\w{W}_1(t) (\w{\varLambda }_1^n(t)-
\w{G}_1^n(t)))'}\DD t
\\
&\quad+\frac{1}{2\pi}\int_{-\pi}^{\pi}\ab{(
\w{G}_1^n(t) (\w{W}_1(t)-\w
{V}_2(t)))'}\DD t +\frac{1}{2\pi}\int
_{-\pi}^{\pi}\ab{(\w{\varLambda }_2(t)
\w{W}_2(t))'}\DD t
\\
&\leqslant C(d+1) \bigl(\gamma\ee^{-0.5Cn\gamma}\ee^{-0.5Cn\gamma}(1+\beta /
\gamma)+n\ee^{n\ln(\beta+4\gamma)}\bigr).
\end{align*}

Hence,
\[
k(1+\beta/\gamma)^{-1}\ab{F_n-(G_1^nV_2+E
)\{k\}}\leqslant\frac{C(d+1)\ee
^{-Cn\gamma}}{n}
\]
and
\[
\ab{F_n-(G_1^nV_2+E)
\{k\}}\leqslant\frac{C(d+1)\ee^{-Cn\gamma}}{n},
\]
since $\ab{M}\leqslant\localnorm{M}\leqslant\norm{M}$.

Summing those inequalities, we get
\[
\ab{F_n-(G_1^nV_2+E)
\{k\}}\leqslant\frac{C(d+1)\ee^{-Cn\gamma
}}{n(1+k(1+\beta/\gamma)^{-1})}=\frac{C(d+1)\ee^{-Cn\gamma}(\beta+\gamma
)}{n(\beta+(k+1)\gamma)}.
\]

In order to prove the second inequality of the theorem, we apply the inversion
formula\index{inversion formula} (\ref{non-uniform_2}) with $a=0$:
\begin{align*}
\hspace{2em}&\hspace{-2em} k\ab{F_n-(G_1^nV_2+E) (k)}\\
&\leqslant\frac{1}{2\pi}\int_{-\pi}^{\pi}\AB{
\biggl(\frac{\w
{W}_1(t)}{\ee^{-\ii t}-1}\bigl(\w{\varLambda}_1^n(t)-
\w{G}_1^n(t)\bigr) \biggr)'}\DD t
\\
&\quad+\frac{1}{2\pi}\int_{-\pi}^{\pi}\AB{ \biggl(
\w{G}_1^n(t) \biggl(\frac
{\w{W}_1(t)}{\ee^{-\ii t}-1}-
\frac{\w{V}_2(t)}{\ee^{-\ii t}-1} \biggr) \biggr)'}\DD t
\\
&\quad+\frac{1}{2\pi}\int_{-\pi}^{\pi}\AB{ \biggl(
\w{\varLambda}_2(t)\frac{\w
{W}_2(t)}{\ee^{-\ii t}-1} \biggr)'}\DD t.
\end{align*}

The summands can be estimated by using the inequalities from the proof of
Theorem \ref{T4}:
\begin{gather*}
\AB{\frac{\w{W}_1(t)}{\ee^{-\ii t}-1}}\ab{(\w{\varLambda}_1^n(t)-
\w {G}_1^n(t))'}\leqslant C(d+1)
\gamma^2\ee^{-Cn\gamma},\\
\frac{\ee^{(d+1)\ii t}-1}{\ee^{-\ii t}-1}=\frac{(\eit-1)(1+\eit+\cdots
+\ee^{d\ii t})}{\ee^{-it}(1-\eit)} =-\eit(1+\eit+\cdots+\deit),
\end{gather*}
\begin{eqnarray*}
\AB{ \biggl(\frac{\w{W}_1(t)}{\ee^{-\ii t}-1} \biggr)'}\!\!\!\!&\leqslant&\!\!\!\!
\frac
{Cd^2}{\gamma^2}, \qquad\AB{ \biggl(\frac{\w{W}_2(t)}{\ee^{-\ii t}-1} \biggr)'}
\leqslant Cd^2,
\\[2pt]
\AB{ \biggl(\frac{\w{W}_1(t)}{\ee^{-\ii t}-1} \biggr)'}\ab{\w{\varLambda
}_1^n(t)-\w{G}_1^n(t)}
\!\!\!\!&\leqslant&\!\!\!\! Cn\gamma^3\ee^{-Cn\gamma}\frac
{d^2}{\gamma^2}\leqslant
Cd^2\ee^{-Cn\gamma},
\\[2pt]
\AB{\w{G}_1^n(t)' \biggl(
\frac{\w{W}_1(t)-\w{V}_2(t)}{\ee^{-\ii t}-1} \biggr)}\!\!\!\!&\leqslant&\!\!\!\! C(d+1)\gamma\ee^{-Cn\gamma},
\\[2pt]
\AB{\w{G}_1^n(t) \biggl(\frac{\w{W}_1(t)-\w{V}_2(t)}{\ee^{-\ii t}-1}
\biggr)'}\!\!\!\!&\leqslant&\!\!\!\! \frac{Cd^2\ee^{-Cn\gamma}}{\gamma^2},
\\[2pt]
\AB{\w{\varLambda}_2^n(t)'
\frac{\w{W}_2(t)}{\ee^{-\ii t}-1}}\!\!\!\!&\leqslant&\!\!\!\!C(d+1)\ee^{-Cn},
\\[2pt]
\AB{\w{\varLambda}_2^n(t) \biggl(\frac{\w{W}_2(t)}{\ee^{-\ii t}-1}
\biggr)'}\!\!\!\!&\leqslant&\!\!\!\! Cd^2(\beta+4\gamma)^n.
\end{eqnarray*}

Thus, we get
\[
k\gamma^2\ab{F_n-(G_1^nV_2+E
) (k)}\leqslant\frac{Cd^2\ee^{-Cn\gamma}}{n}
\]
and
\[
\ab{F_n-(G_1^nV_2+E)
(k)}\leqslant\frac{C(d+1)\ee^{-Cn\gamma}}{n}.
\]

By summing the above inequalities we arrive at
\[
\ab{F_n-(G_1^nV_2+E)
(k)}\leqslant\frac{Cd^2\ee^{-Cn\gamma}}{n(1+k\gamma^2)}.\qedhere
\]
\end{proof}


\begin{thebibliography}{21}

\bibitem{BL06}
\begin{barticle}
\bauthor{\bsnm{Barbour}, \binits{A.D.}},
\bauthor{\bsnm{Lindvall}, \binits{T.}}:
\batitle{Translated Poisson approximation for Markov chains}.
\bjtitle{J. Theor. Probab.}
\bvolume{19}(\bissue{3}),
\bfpage{609}--\blpage{630}
(\byear{2006}).
\bid{doi={10.1007/\\s10959-006-0047-9}, mr={2280512}}
\end{barticle}
%
\OrigBibText
Barbour,~A.D., Lindvall~T.: Translated Poisson approximation for Markov chains.
J. Theor. Probab. \textbf{19}(3), 609--630,
(2006). MR2280512, https://doi.org/10.1007/s10959-006-0047-9
\endOrigBibText
\bptok{structpyb}
\endbibitem

\bibitem{C16}
\begin{bbook}
\bauthor{\bsnm{\v{C}ekanavi\v{c}ius}, \binits{V.}}:
\bbtitle{Approximation methods in probability theory}.
\bsertitle{Universitext},
\bpublisher{Springer}
(\byear{2016}).
\bid{doi={10.1007/978-3-319-34072-2}, mr={3467748}}
\end{bbook}
%
\OrigBibText
\v{C}ekanavi\v{c}ius,~V.:
{Approximation methods in probability theory}.
Universitext,
Springer-Verlag,
(2016). MR3467748. https://doi.org/10.1007/978-3-319-34072-2
\endOrigBibText
\bptok{structpyb}
\endbibitem

\bibitem{CR09}
\begin{barticle}
\bauthor{\bsnm{\v{C}ekanavi\v{c}ius}, \binits{V.}},
\bauthor{\bsnm{Roos}, \binits{B.}}:
\batitle{Poisson type approximations for the Markov binomial distribution}.
\bjtitle{Stoch. Process. Appl.}
\bvolume{119},
\bfpage{190}--\blpage{207}
(\byear{2009}).
\bid{doi={10.1016/j.spa.2008.01.008}, mr={2485024}}
\end{barticle}
%
\OrigBibText
\v{C}ekanavi\v{c}ius,~V., Roos,~B.:
{Poisson type approximations for the Markov binomial distribution}.
Stochastic Proc. Appl. \textbf{119}, 190--207, (2009). MR2485024 (2010d:60161), https://doi.org/10.1016/j.spa.2008.01.008
\endOrigBibText
\bptok{structpyb}
\endbibitem

\bibitem{CV10}
\begin{barticle}
\bauthor{\bsnm{\v{C}ekanavi\v{c}ius}, \binits{V.}},
\bauthor{\bsnm{Vellaisamy}, \binits{P.}}:
\batitle{Compound Poisson and signed compound Poisson approximations to the Markov binomial law}.
\bjtitle{Bernoulli}
\bvolume{16}(\bissue{4}),
\bfpage{1114}--\blpage{1136}
(\byear{2010}).
\bid{doi={10.3150/09-BEJ246}, mr={2759171}}
\end{barticle}
%
\OrigBibText
\v{C}ekanavi\v{c}ius,~V., Vellaisamy~P.:
{Compound Poisson and signed compound Poisson approximations to the Markov binomial law}.
Bernoulli. \textbf{16}(4), 1114-1136, (2010). MR2759171 (2012c:60048), https://doi:10.3150/09-BEJ246
\endOrigBibText
\bptok{structpyb}
\endbibitem

\bibitem{DD92}
\begin{barticle}
\bauthor{\bparticle{De} \bsnm{Pril}, \binits{N.}},
\bauthor{\bsnm{Dhaene}, \binits{J.}}:
\batitle{Error bounds for compound Poisson approximations of the individual risk model}.
\bjtitle{ASTIN Bull.}
\bvolume{22}(\bissue{2}),
\bfpage{135}--\blpage{148}
(\byear{1992}).
\bid{doi={10.2143/\\AST.22.2.2005111}}
\end{barticle}
%
\OrigBibText
De Pril,~N., Dhaene,~J.:
{Error bounds for compound Poisson approximations of the individual risk model.}
Astin Bulletin. \textbf{22}(2), 135--148, (1992). https://doi.org/10.2143/AST.22.2.2005111
\endOrigBibText
\bptok{structpyb}
\endbibitem

\bibitem{E99}
\begin{barticle}
\bauthor{\bsnm{Erhardsson}, \binits{T.}}:
\batitle{Compound Poisson approximation for Markov chains using Stein's method}.
\bjtitle{Ann. Probab.}
\bvolume{27}(\bissue{1}),
\bfpage{565}--\blpage{596}
(\byear{1999}).
\bid{doi={10.1214/aop/\\1022677272}, mr={1681149}}
\end{barticle}
%
\OrigBibText
Erhardsson,~T.:
{Compound Poisson approximation for Markov chains using Stein's method.}
Ann.Probab. \textbf{27}(1), 565--596, (1999). MR1681149 (2000b:60048), https://doi:10.1214/aop/1022677272
\endOrigBibText
\bptok{structpyb}
\endbibitem

\bibitem{Gan82}
\begin{barticle}
\bauthor{\bsnm{Gani}, \binits{J.}}:
\batitle{On the probability generating function of the sum of Markov-Bernoulli random variables}.
\bjtitle{J. Appl. Probab.}
(\bcomment{Special vol.})
\bvolume{19A},
\bfpage{321}--\blpage{326}
(\byear{1982}).
\bid{doi={10.2307/3213571}, mr={0633201}}
\end{barticle}
%
\OrigBibText
{Gani,~J.}:
 On the probability generating function of the sum of
 Markov-Bernoulli random variables.
 { J. Appl. Probab.} (Special vol.) \textbf{19A,} 321--326 (1982). MR0633201
https://doi.org/10.2307/3213571
\endOrigBibText
\bptok{structpyb}
\endbibitem

\bibitem{Ge84}
\begin{barticle}
\bauthor{\bsnm{Gerber}, \binits{H.U.}}:
\batitle{Error bounds for the compound Poisson approximation}.
\bjtitle{Insur. Math. Econ.}
\bvolume{3},
\bfpage{191}--\blpage{194}
(\byear{1984}).
\bid{doi={10.1016/0167-6687(84)90062-3}, mr={0752200}}
\end{barticle}
%
\OrigBibText
Gerber, H.U.: Error bounds for the compound Poisson approximation.
Insurance: Mathematics and Economics. \textbf{3}, 191--194, (1984). MR0752200 (86b:62162), https://doi.org/10.1016/0167-6687(84)90062-3
\endOrigBibText
\bptok{structpyb}
\endbibitem

\bibitem{H85}
\begin{barticle}
\bauthor{\bsnm{Hipp}, \binits{C.}}:
\batitle{Approximation of aggregate claims distributions by compound Poisson distribution}.
\bjtitle{Insur. Math. Econ.}
\bvolume{4}(\bissue{4}),
\bfpage{227}--\blpage{232}
(\byear{1985}).
\bid{doi={10.1016/\\0167-6687(85)90032-0}, mr={0810720}}
\end{barticle}
%
\OrigBibText
Hipp,~C.:
{Approximation of aggregate claims distributions by compound Poisson distribution.}
Insurance: Mathematics and Economics. \textbf{4}(4), 227--232, (1985). MR0810720 (86m:62185), https://doi.org/10.1016/0167-6687(85)90032-0
\endOrigBibText
\bptok{structpyb}
\endbibitem

\bibitem{Aki93}
\begin{barticle}
\bauthor{\bsnm{Hirano}, \binits{K.}},
\bauthor{\bsnm{Aki}, \binits{S.}}:
\batitle{On number of success runs of specified length in a two-state Markov chain}.
\bjtitle{Stat. Sin.}
\bvolume{3},
\bfpage{313}--\blpage{320}
(\byear{1993}).
\bid{doi={10.1239/aap/\\1029955143}, mr={1243389}}
\end{barticle}
%
\OrigBibText
{Hirano,~K., Aki,~S.}: On number of success
runs of specified length in a two-state Markov chain. { Statist.
Sinica.} \textbf{3}, 313--320 (1993). MR1243389,
https://doi.org/10.1239/aap/1029955143
\endOrigBibText
\bptok{structpyb}
\endbibitem

\bibitem{LS17}
\begin{barticle}
\bauthor{\bsnm{Leipus}, \binits{R.}},
\bauthor{\bsnm{\v{S}iaulys}, \binits{J.}}:
\batitle{On the random max-closure for heavy-tailed random variables}.
\bjtitle{Lith. Math. J.}
\bvolume{57}(\bissue{2}),
\bfpage{208}--\blpage{221}
(\byear{2017}).
\bid{doi={10.1007/\\s10986-017-9355-2}, mr={3654985}}
\end{barticle}
%
\OrigBibText
Leipus,~R., \v{S}iaulys,~J.: On the random max-closure for heavy-tailed random variables.
Lithuanian Math. J. \textbf{57}(2), 208--221, (2017). MR3654985, https://doi.org/10.1007/s10986-017-9355-2
\endOrigBibText
\bptok{structpyb}
\endbibitem

\bibitem{Pi04}
\begin{barticle}
\bauthor{\bsnm{Pitts}, \binits{S.M.}}:
\batitle{A functional approach to approximations for the individual risk model}.
\bjtitle{ASTIN Bull.}
\bvolume{34},
\bfpage{379}--\blpage{397}
(\byear{2004}).
\bid{doi={10.1017/\\S051503610001374X}, mr={2086451}}
\end{barticle}
%
\OrigBibText
Pitts,~S.M.: A functional approach to approximations for the individual risk model. Astin Bulletin. \textbf{34}, 379--397, (2004).
MR2086451, https://doi.org/10.1017/S051503610001374X
\endOrigBibText
\bptok{structpyb}
\endbibitem

\bibitem{Pre85}
\begin{barticle}
\bauthor{\bsnm{Presman}, \binits{E.L.}}:
\batitle{Approximation in variation of the distribution of a sum of independent Bernoulli variables with a Poisson law}.
\bjtitle{Theory Probab. Appl.}
\bvolume{30}(\bissue{2}),
\bfpage{417}--\blpage{422}
(\byear{1986}).
\bid{doi={10.1137/1130051}, mr={0792634}}
\end{barticle}
%
\OrigBibText
Presman, E.L.: Approximation in variation of the distribution of a sum of
independent Bernoulli variables with a Poisson law. Theory
Probab. Appl. \textbf{30}(2), 417--422, (1986). MR0792634 (87c:60027),
\endOrigBibText
\bptok{structpyb}
\endbibitem

\bibitem{Ro07}
\begin{barticle}
\bauthor{\bsnm{Roos}, \binits{B.}}:
\batitle{On variational bounds in the compound Poisson approximation of the individual risk model}.
\bjtitle{Insur. Math. Econ.}
\bvolume{40},
\bfpage{403}--\blpage{414}
(\byear{2007}).
\bid{doi={\\10.1016/j.insmatheco.2006.06.003}, mr={2310979}}
\end{barticle}
%
\OrigBibText
Roos,~B.: {On variational bounds in the compound Poisson approximation of
the individual risk model.} Insurance: Mathematics and Economics. \textbf{40}, 403--414,
(2007). MR2310979 (2009e:62424), https://doi.org/10.1016/j.insmatheco.2006.06.003
\endOrigBibText
\bptok{structpyb}
\endbibitem

\bibitem{SC15}
\begin{barticle}
\bauthor{\bsnm{\v{S}liogere}, \binits{J.}},
\bauthor{\bsnm{\v{C}ekanavi\v{c}ius}, \binits{V.}}:
\batitle{Two limit theorems for Markov binomial distribution}.
\bjtitle{Lith. Math. J.}
\bvolume{55}(\bissue{3}),
\bfpage{451}--\blpage{463}
(\byear{2015}).
\bid{doi={10.1007/\\s10986-015-9291-y}, mr={3379037}}
\end{barticle}
%
\OrigBibText
\v{S}liogere,~J., \v{C}ekanavi\v{c}ius,~V.:
{Two limit theorems for Markov binomial distribution.}
Lithuanian Math. J. \textbf{55}(3), 451--463,
(2015). MR3379037, https://doi.org/10.1007/s10986-015-9291-y
\endOrigBibText
\bptok{structpyb}
\endbibitem

\bibitem{SC16}
\begin{barticle}
\bauthor{\bsnm{\v{S}liogere}, \binits{J.}},
\bauthor{\bsnm{\v{C}ekanavi\v{c}ius}, \binits{V.}}:
\batitle{Approximation of symmetric three-state Markov chain by compound Poisson law}.
\bjtitle{Lith. Math. J.}
\bvolume{56}(\bissue{3}),
\bfpage{417}--\blpage{438}
(\byear{2016}).
\bid{doi={10.1007/s10986-016-9326-z}, mr={3530227}}
\end{barticle}
%
\OrigBibText
\v{S}liogere,~J., \v{C}ekanavi\v{c}ius,~V.:
{Approximation of symmetric three-state Markov chain by compound Poisson law.}
Lithuanian Math. J. \textbf{56}(3), 417--438,
(2016).
MR3530227, https://doi.org/10.1007/s10986-016-9326-z
\endOrigBibText
\bptok{structpyb}
\endbibitem

\bibitem{WGYCh18}
\begin{barticle}
\bauthor{\bsnm{Wang}, \binits{K.}},
\bauthor{\bsnm{Gao}, \binits{M.}},
\bauthor{\bsnm{Yang}, \binits{Y.}},
\bauthor{\bsnm{Chen}, \binits{Y.}}:
\batitle{Asymptotics for the finite-time ruin probability in a discrete-time risk model with dependent insurance and financial risks}.
\bjtitle{Lith. Math. J.}
\bvolume{58}(\bissue{1}),
\bfpage{113}--\blpage{125}
(\byear{2018}).
\bid{doi={10.1007/\\s10986-017-9378-8}, mr={3779067}}
\end{barticle}
%
\OrigBibText
Wang,~K., Gao,~M., Yang,~Y., Chen,~Y.: Asymptotics for the finite-time ruin probability in a discrete-time
risk model with dependent insurance and financial risks. Lithuanian Math. J.
 \textbf{58}(1), 113--125,
(2018). MR3779067, https://doi.org/10.1007/s10986-017-9378-8
\endOrigBibText
\bptok{structpyb}
\endbibitem

\bibitem{XZ09}
\begin{barticle}
\bauthor{\bsnm{Xia}, \binits{A.}},
\bauthor{\bsnm{Zhang}, \binits{M.}}:
\batitle{On approximation of Markov binomial distributions}.
\bjtitle{Bernoulli}
\bvolume{15},
\bfpage{1335}--\blpage{1350}
(\byear{2009}).
\bid{doi={10.3150/09-BEJ194}, mr={2597595}}
\end{barticle}
%
\OrigBibText
Xia,~A., Zhang~M.:
{On approximation of Markov binomial distributions.}
Bernoulli. \textbf{15}, 1335--1350, (2009). MR2597595 (2011e:60168), https://doi:10.3150/09-BEJ194
\endOrigBibText
\bptok{structpyb}
\endbibitem

\bibitem{YY10}
\begin{barticle}
\bauthor{\bsnm{Yang}, \binits{G.}},
\bauthor{\bsnm{Miao}, \binits{Y.}}:
\batitle{Moderate and Large Deviation Estimate for the Markov-Binomial Distribution}.
\bjtitle{Acta Appl. Math.}
\bvolume{110},
\bfpage{737}--\blpage{747}
(\byear{2010}).
\bid{doi={10.1007/s10440-009-9471-z}, mr={2610590}}
\end{barticle}
%
\OrigBibText
Yang,~G., Miao~Y.: Moderate and Large Deviation Estimate
for the Markov-Binomial Distribution. Acta Appl. Math. \textbf{110}, 737--747, (2010).
MR2610590 , https://doi.org/10.1007/s10440-009-9471-z
\endOrigBibText
\bptok{structpyb}
\endbibitem

\bibitem{YW13}
\begin{barticle}
\bauthor{\bsnm{Yang}, \binits{Y.}},
\bauthor{\bsnm{Wang}, \binits{Y.}}:
\batitle{Tail behavior of the product of two dependent random variables with applications to risk theory}.
\bjtitle{Extremes}
\bvolume{16}(\bissue{1}),
\bfpage{55}--\blpage{74}
(\byear{2013}).
\bid{doi={10.1007/s10687-012-0153-2}, mr={3020177}}
\end{barticle}
%
\OrigBibText
Yang,~Y., Wang,~Y.: Tail behavior of the product of t
wo dependent random variables with applications to risk
theory. Extremes. \textbf{16}(1), 55--74, (2013). MR3020177, https://doi.org/10.1007/s10687-012-0153-2
\endOrigBibText
\bptok{structpyb}
\endbibitem

\bibitem{ZLL14}
\begin{barticle}
\bauthor{\bsnm{Zhang}, \binits{H.}},
\bauthor{\bsnm{Liu}, \binits{Y.}},
\bauthor{\bsnm{Li}, \binits{B.}}:
\batitle{Notes on discrete compound Poisson model with applications to risk theory}.
\bjtitle{Insur. Math. Econ.}
\bvolume{59},
\bfpage{325}--\blpage{336}
(\byear{2014}).
\bid{doi={10.1016/j.insmatheco.2014.09.012}, mr={3283233}}
\end{barticle}
%
\OrigBibText
Zhang,~H., Liu,~Y., Li,~B.: Notes on discrete compound Poisson model with applications to
risk theory. Insurance: Mathematics and Economics. \textbf{59}, 325--336, (2014). MR3283233, http://dx.doi.org/10.1016/j.insmatheco.2014.09.012
\endOrigBibText
\bptok{structpyb}
\endbibitem

\end{thebibliography}
%

%




\end{document}